\documentclass[12 pt]{article}%
\usepackage{amsmath, amsfonts, amsthm, color,latexsym}
\usepackage{amsmath, ulem}
\usepackage{amsfonts}
\usepackage{amssymb}
\usepackage{color, soul}
\usepackage{xcolor}
\usepackage{tikz-cd}
\usepackage[all]{xy}
\usepackage{graphicx}%
\providecommand{\U}[1]{\protect\rule{.1in}{.1in}}
\allowdisplaybreaks[4]
\newtheorem{theorem}{Theorem}[section]
\newtheorem{proposition}[theorem]{Proposition}
\newtheorem{corollary}[theorem]{Corollary}
\newtheorem{example}[theorem]{Example}
\newtheorem{examples}[theorem]{Examples}
\newtheorem{remark}[theorem]{Remark}

\newtheorem{lemma}[theorem]{Lemma}
\newtheorem{final remark}[theorem]{Final Remark}
\newtheorem{definition}[theorem]{Definition}
\allowdisplaybreaks[4]

\newcommand{\cvf}{\overset{\omega}{\rightarrow}}

\newcommand {\R}{\mathbb{R}}
\newcommand {\K} {\mathbb{K}}

\newcommand {\N} {\mathbb{N}}
\newcommand{\norma}[1]{\| #1 \|}
\newcommand{\conj}[2]{\left \{ {#1} \, : \, {#2} \right \}}

\begin{document}

\title{Norm attainment for multilinear operators and polynomials on Banach Spaces and Banach lattices}
\author{Luis A. Garcia\thanks{Supported by Fapemig (Grant APQ-02622-26)},\,  Jos\'e Lucas P. Luiz\thanks{Supported by Fapemig (Grants APQ-02622-26 and APQ-01853-23).}\, and  Vinícius Miranda\thanks{Supported by FAPESP (Grants 2023/12916-1 and 2025/08630-0) and Fapemig (Grants APQ-02622-26 and APQ-01853-23). \newline 2020 Mathematics Subject Classification: 46G25, 46B42, 46B28. \newline Keywords: Norm-attaining multilinear operators, Norm-attaining polynomials, Banach lattices.}}
\date{}
\maketitle

\begin{abstract}
Let $X_1,\dots,X_n$ be reflexive Banach spaces and let $Y$ be a Banach space. Under suitable approximation properties, we prove that every continuous $n$-linear operator from $X_1\times\cdots\times X_n$ into $Y$ attains its norm if and only if every such operator is weakly sequentially continuous. The same equivalence holds for continuous $n$-homogeneous polynomials.
We also establish the corresponding equivalences for positive multilinear operators and homogeneous polynomials between Banach lattices. In particular, every positive bilinear operator from $\ell_4\times\ell_4$ into $L_1([0,1])$ attains its norm, whereas some bilinear operators between the same spaces fail to do so.
\end{abstract}

\section{Introduction}

The study of norm-attaining linear operators plays a central role in the theory of operators between Banach spaces. The problem dates back to the classical works of James \cite{james}, Bishop and Phelps \cite{bishophelps}, and Lindenstrauss \cite{lind63}, which establish deep connections between the geometry of Banach spaces and the behavior of their operators. James’ theorem, for instance, characterizes the reflexivity of a space $X$ by the property that every continuous linear functional on $X$ attains its norm. The Bishop-Phelps Theorem, in turn, ensures that the set of norm-attaining continuous linear functionals on a space $X$ is dense in $X^*$.

A natural line of investigation arising from these results concerns the extent to which norm attainment properties extend from functionals to linear operators. Given Banach spaces $X$ and $Y$, one may ask when the set 
$$ NA(X; Y) = \left \{ T \in \mathcal{L}(X; Y) : T \text{ attains its norm} \right \} $$
is dense in ${\cal L}(X; Y)$ (see, e.g., \cite{bishop}). Beyond density, a stronger and more rigid phenomenon emerges when one asks whether every operator attains its norm, that is, when $NA(X;Y) = {\cal L}(X; Y)$. Two fundamental contributions in this direction are due to Holub \cite{holub} and, later, Mujica \cite{mujica}, who established structural conditions ensuring this equality. This principle was recently extended by Dantas, Jung, and Martínez-Cervantes in \cite[Theorem B]{sheldon}. Their result shows that if $X$ is reflexive and the pair $(X, Y)$ satisfies the bounded  compact approximation property, then every operator $T: X \to Y$ is compact if and only if every operator $T: X \to Y$ attains its norm.

Our first goal is to extend this result to multilinear operators. The passage from the linear to the multilinear setting is, however, far from straightforward. As a first indication of this additional complexity, one observes that the direct analogue of James’ theorem fails dramatically in the bilinear framework. To see this, notice that the projective tensor product $\ell_2 \, \widehat{\otimes}_{\pi} \ell_2$ is not reflexive (see \cite[Corollary 4.24]{ryan}), so there exists a continuous linear functional $\varphi: \ell_2 \hat{\otimes}_{\pi} \ell_2 \to \K$ that does not attain its norm. If $A: \ell_2 \times \ell_2 \to \K$ is the bilinear form given by $A(x,y) = \varphi(x \otimes y)$, then $A$ also fails to attain its norm, since any norm-attaining pair for $A$ would automatically yield a norm-attaining point for 
$\varphi$. Motivated by this contrast, we investigate conditions ensuring universal norm attainment for multilinear mappings. More precisely, we establish that under suitable approximation properties, every $n$-linear operator from $X_1 \times \cdots \times X_n$ into $Y$ attains its norm if and only if every such operator is weakly sequentially continuous. These conditions are also characterized in terms of weak sequential continuity of scalar-valued forms and compactness of $Y$-valued operators. We also establish the corresponding result for homogeneous polynomials in Section 2. Comparing the multilinear and polynomial settings leads to one of the questions discussed in Section 4.

In the context of Banach lattices, however, the situation becomes considerably more subtle. For instance, a lattice version of James' theorem asserts that a Banach lattice $E$ with order-continuous norm is reflexive if and only if every positive linear functional on $E$ attains its norm. This result was obtained independently by Ji, Lee, and Bu in \cite{bujames} and by Oikhberg and Tursi in \cite{oikhberg}. It is important to emphasize that the additional assumption of order continuity of the norm is essential: indeed, in $C(K)$ every positive linear functional attains its norm at the constant function equal to $1$ (see \cite[Exercise 2, p.270]{positiveoperators}), although the space is not reflexive. From the viewpoint of linear operators, the second and third named authors \cite{luizmiranda} established a positive-operator counterpart of the theorem of Dantas, Jung, and Martínez-Cervantes (see \cite{luizmiranda}). Related work on positive versions of the Bishop-Phelps-Bollobás property \cite{acosta1,acosta4} and on norm-attaining Riesz homomorphisms \cite{bilojfa,dantas2022} further illustrates that classical norm-attainment phenomena do not transfer directly to Banach lattices.

Section 3 develops positive counterparts of our multilinear and polynomial results. For this purpose, we introduce the sequentially positive approximation property and prove that a Banach lattice has this property if and only if it is separable, atomic, and has order continuous norm. We then use this property to characterize universal norm attainment for positive multilinear operators and homogeneous polynomials in terms of weak sequential continuity. As a concrete consequence, every positive bilinear operator from $\ell_4\times\ell_4$ into $L_1([0,1])$ attains its norm, whereas some bilinear operators between the same spaces fail to do so. Section 4 concludes the paper with open questions.

Before proceeding, we introduce some definitions and notation. Throughout this paper, $\K$ denotes $\R$ or $\mathbb{C}$, $X$ and $Y$ denote Banach spaces over $\K$, and $E$ and $F$ denote (real) Banach lattices. The symbols ${\rm id}_X$, $E^+$, $B_X$, and $S_X$ denote, respectively, the identity operator on $X$, the positive cone of $E$, the closed unit ball of $X$, and the unit sphere of $X$.
Given Banach spaces $X_1, \dots, X_n$ and $Y$, the Banach space of all continuous $n$-linear operators $A\colon X_1 \times \cdots \times X_n \to Y$ is denoted by $\mathcal{L}(X_1, \dots, X_n; Y)$. In the case that $F = \K$, we simply write $\mathcal{L}(X_1, \dots, X_n)$. The norm of $A$ is defined to be
$$ \norma{A} = \sup \conj{\norma{A(x_1, \dots, x_n)}}{x_1 \in B_{X_1}, \dots, x_n \in B_{X_n}}. $$
Given Banach lattices $E_1, \dots, E_n$ and $F$, a $n$-linear operator $A\colon E_1 \times \cdots \times E_n \to F$ is said to be positive if $A(x_1, \dots, x_n) \geq 0$ for all $x_1 \in E_1^+, \dots, x_n \in E_n^+$. The difference of two positive $n$-linear operators is called a regular $n$-linear operator, and the set of all regular $n$-linear operators from $E_1 \times \cdots \times E_n $ into $F$ is denoted by $\mathcal{L}^r(E_1, \dots, E_n; F)$. Whenever $F$ is Dedekind complete, $\mathcal{L}^r(E_1, \dots, E_n; F)$ is a Banach lattice with the regular norm $\norma{A}_r = \norma{|A|}$, where $|A|$ denotes the absolute value of the regular $n$-linear operator $A\colon E_1 \times \cdots \times E_n \to F$.

A mapping $P:X\to Y$ is an $n$-homogeneous continuous polynomial if there exists a unique symmetric continuous $n$-linear operator $T_P:X\times\cdots\times X\to Y$ such that
$P(x)=T_P(x,\ldots,x)$ for every $x\in X.$
The space of all such polynomials, endowed with the sup norm, is denoted by $\mathcal P(^n X;Y)$. 
Moreover, the norms of $P$ and its associated multilinear mapping $T_P$ satisfy $$\displaystyle \|P\|\le \|T_P\|\le \frac{n^n}{n!}\|P\|.$$
In the setting of Banach lattices, an $n$-homogeneous polynomial $P\colon E \to F$ is positive  if its associated symmetric multilinear operator $T_P\colon E^n \to F$ is positive. The difference of two positive $n$-homogeneous polynomials is called a regular homogeneous polynomial, and the set of such
polynomials is denoted by $\mathcal{P}^r(^nE; F)$. When $F$ is the scalar field we simply write $\mathcal{P}^r(^nE)$. If $E$ and $F$ are Banach lattices with $F$ Dedekind complete, then $\mathcal{P}^r(^nE; F)$ is a Banach lattice with the regular norm $\norma{P}_r = \norma{|P|}$, where $|P|$ denotes the absolute value of the regular $n$-homogeneous polynomial $P\colon E \to F$. 

 We refer the reader to \cite{positiveoperators, meyer} for background on Banach lattices, to \cite{albiac} for Banach space theory, to \cite{dineen, mujica} for polynomials on Banach spaces, to 
 \cite{ryan} for tensor products between Banach spaces, to \cite{bubuskes, loane} for positive/regular multilinear operators or polynomials, and to \cite{bubuskes, frem2}
for tensor products between vector lattices.

\section{The Banach space setting}\label{section2}

As announced in the Introduction, our first objective is to provide a multilinear version of \cite[Theorem B]{sheldon}. To address this problem, we begin by examining how the proof of that result can be adapted to the multilinear setting. A key starting point is the observation that, in multilinear extensions of results involving compact linear operators, compactness is naturally replaced by the class of weakly sequentially continuous multilinear mappings.
To this end, we recall the following definitions. An $n$-linear continuous operator $A: X_1 \times \cdots \times X_n \to Y$ is said to \\
(i) be {\bf norm-attaining} if there are $x_1 \in B_{X_1}, \dots, x_n \in B_{X_n}$ such that
$$ \norma{A(x_1, \dots, x_n)} = \norma{A}.$$
\noindent (ii) be {\bf weakly sequentially continuous} if $A(x_{1,k}, \dots, x_{n,k}) \to A(x_1, \dots, x_n)$ in $Y$ whenever $(x_{j,k})_k \subset X_j$ converges to $x_j \in X_j$ in the weak topology for every $j =1, \dots, n$. 
The collection of such operators is denoted by $\mathcal{L}_{wsc} (X_1, \dots, X_n; Y)$. \\
\noindent (iii) be {\bf compact} if $A(B_{X_1} \times \cdots \times B_{X_n})$ is a relatively compact subset of $Y$. The collection of all such operators is denoted by $\mathcal{K}(X_1, \ldots, X_n; Y)$. \\
\noindent (iv) have {\bf finite type} if there are vectors $y_1, \dots, y_k \in Y$ and linear continuous functionals $\varphi_{j,1}, \dots, \varphi_{j,k} \in X_j^*$ for every $j =1, \dots, n$ such that
$$ A(x_1, \dots, x_n) = \sum_{i=1}^k \varphi_{1,i}(x_1)\cdots \varphi_{n,i}(x_n) y_i  $$
for all $x_1 \in X_1, \dots, x_n \in X_n$. The collection of all  $n$-linear operators with finite type from $X_1 \times \cdots \times X_n$ into $Y$ is denoted by $\mathcal{F}(X_1, \dots, X_n; Y)$.


\medskip

The main result of this section is stated as follows.

\begin{theorem} \label{maintheo}
    Let $X_1, \dots, X_n$ be reflexive Banach spaces and $Y$ be a Banach space.  Assume that for every $A \in \mathcal{L}(X_1, \dots, X_n; Y)$, there exists $\lambda = \lambda(A) \geq 1$ such that  for all compact sets $K_1 \subset X_1, \dots, K_n \subset X_n$ and every $\varepsilon > 0$ there exists $B \in \mathcal{F}(X_1, \dots, X_n; Y)$ with $\norma{B} \leq \lambda$ such that
    $$ \norma{A(x_1, \dots, x_n) - B(x_1, \dots, x_n)} < \varepsilon  $$
    for all $x_1 \in K_1,\dots x_n \in K_n$.
    Then,  the following are equivalent: \\
    {\rm (1)} Every $n$-linear continuous operator $A: X_1 \times \cdots \times X_n \to Y$ is weakly sequentially continuous. \\
    {\rm (2)} Every $n$-linear continuous operator $A: X_1 \times \cdots \times X_n \to Y$ is norm-attaining. \\
    {\rm (3)} Every $n$-linear continuous form $\varphi: X_1 \times \cdots \times X_n \to \K$ is weakly sequentially continuous and every $n$-linear continuous operator $A: X_1 \times \cdots \times X_n \to Y$ is compact.
\end{theorem}


To prove Theorem \ref{maintheo}, we begin by recalling the definition of the {\bf strong operator topology} ($SOT$, for short) on $\mathcal{L}(X; Y)$. Its basic neighborhoods are of the form
$$N(T; F,\varepsilon)=\conj{S \in \mathcal{L}(X; Y)}{\norma{T(x) - S(x)} < \varepsilon \, \text{ for every } x \in F},
$$
where $\varepsilon > 0$, $F \subset X$ is a finite set, and $T\in \mathcal{L}(X;Y)$. 

In Remark \ref{remark} below, we will provide  a characterization of convergence in the strong operator topology for spaces of operators defined on the projective tensor products. For this purpose, we recall that the
$n$-fold projective tensor product $X_1 \widehat{\otimes}_\pi \cdots \widehat{\otimes}_\pi X_n$ is defined as the completion of the algebraic tensor product $X_1 \otimes \cdots \otimes X_n$ with respect to the projective tensor norm
$$
\|u\|_{\pi} := \inf \left\{ \sum_{i=1}^k \|x_{1,i}\|\cdots\|x_{n,i}\| : 
u = \sum_{i=1}^k x_{1,i}\otimes \cdots \otimes x_{n,i} \right\}.
$$
In this framework, every continuous $n$-linear operator $A:X_1 \times \cdots \times X_n \to Y$ admits a unique linearization
$A^\otimes : X_1 \widehat{\otimes}_\pi \cdots \widehat{\otimes}_\pi X_n \to Y$
satisfying
$ A(x_1,\dots,x_n)=A^\otimes(x_1\otimes \cdots \otimes x_n)$
for all $x_1\in X_1,\dots,x_n\in X_n$. Moreover, the correspondence
$A \mapsto A^\otimes$
defines an isometric isomorphism between $\mathcal{L}(X_1,\dots,X_n;Y)$ and $\mathcal{L}(X_1 \widehat{\otimes}_\pi \cdots \widehat{\otimes}_\pi X_n;Y)$. 

\begin{remark} \label{remark}
    Let $X_1, \dots, X_n$ and $Y$ be Banach spaces, $(A_\alpha)_{\alpha\in \Lambda}$ be a norm bounded net in $\mathcal{L}(X_1,\dots,X_n;Y)$, and $A \in \mathcal{L}(X_1,\dots,X_n;Y)$. Then, $A_\alpha^\otimes \to A^\otimes$ in the $SOT$ of $\mathcal{L}(X_1 \widehat{\otimes}_\pi \cdots \widehat{\otimes}_\pi X_n; Y)$ if and only if $(A_\alpha)_{\alpha\in \Lambda}$ converges pointwise to $A$, that is
$$ \norma{A_\alpha(x_1, \dots, x_n) - A(x_1, \dots, x_n)} \to 0 $$
for all $x_1 \in X_1, \dots, x_n \in X_n$. {\rm To see this, assume that $(A_\alpha)_{\alpha\in\Lambda}$ converges pointwise to $A$
and define $\displaystyle M: = \sup_{\alpha} \norma{A_\alpha} < \infty$. Let $z \in X_1 \widehat{\otimes}_\pi \cdots \widehat{\otimes}_\pi X_n$ and $\varepsilon >0$ be given. Since $X_1 \otimes \cdots \otimes X_n$ is dense in $X_1 \widehat{\otimes}_\pi \cdots \widehat{\otimes}_\pi X_n$, there exists $w := \displaystyle \sum_{i=1}^k x_{1,i} \otimes \cdots \otimes x_{n,i}$ such that
$$ \norma{z - w}_{\pi} < \frac{\varepsilon}{2(\norma{A} + M + 1)}. $$
Since
$ \displaystyle\lim_{\alpha} \norma{A_\alpha(x_{1,i} , \dots ,x_{n,i}) - A(x_{1,i} , \dots ,x_{n,i})} = 0 $
for every $i=1, \dots, k$, there is $\alpha_0 \in \Lambda$ such that
$$ \norma{A_\alpha(x_{1,i} , \dots ,x_{n,i}) - A(x_{1,i} , \dots ,x_{n,i})} < \frac{\varepsilon}{2k} $$
for every $\alpha \geq \alpha_0$ and every $i = 1, \dots, k$. Hence
\begin{align*}
    \norma{A_\alpha^\otimes(z) - A^\otimes(z)} & \leq \norma{A_\alpha^\otimes(z) - A_\alpha^\otimes(w)} + \norma{A_\alpha^\otimes(w) - A^\otimes(w)} + \norma{A^\otimes(w) - A^\otimes(z)} \\
    & \leq [\norma{A_\alpha^\otimes} + \norma{A^\otimes}] \norma{z-w}_{\pi} + \sum_{i=1}^k \norma{A_\alpha(x_{1,i} , \dots ,x_{n,i}) - A(x_{1,i} , \dots ,x_{n,i})} \\
    & < \varepsilon
\end{align*}
for all  $\alpha \geq \alpha_0$. Hence $A_\alpha^\otimes \to A^\otimes$ in the $SOT$ of $\mathcal{L}(X_1\widehat{\otimes}_{\pi} \cdots \widehat{\otimes}_{\pi} X_n; Y)$. The converse is immediate.}
\end{remark}

The next two results are obtained by adapting the arguments of Theorem A and Lemma 3.4 in \cite{sheldon} to the multilinear setting. For a subset $C \subset \mathcal{L}(X_1, \dots, X_n; Y)$, we define 
$ C^\otimes := \conj{A^\otimes}{A \in C}. $

\begin{lemma} \label{lema1}
Let $X_1, \dots, X_n$ and $Y$ be Banach spaces and let $C \subset \mathcal{L}(X_1, \dots, X_n; Y)$ be a norm-closed convex subset. If $C^\otimes$ is not sequentially $SOT$-closed in ${\cal L}(X_1 \widehat{\otimes}_{\pi} \cdots \widehat{\otimes}_{\pi} X_n; Y)$,  then  there exists a non-norm attaining $n$-linear operator $A: X_1 \times \cdots \times X_n \to Y$.
\end{lemma}

\begin{proof} It is easy to see that $C^\otimes$ is a norm-closed convex subset of ${\cal L}(X_1\widehat{\otimes}_{\pi} \cdots \widehat{\otimes}_{\pi}X_n; Y)$ as $C$ is a norm-closed convex subset of ${\cal L}(X_1, \dots, X_n; Y)$. So, since $C^\otimes$ is not sequentially $SOT$-closed, there exists, by \cite[Theorem A]{sheldon}, a non-norm attaining linear operator $T: X_1\widehat{\otimes}_{\pi} \cdots \widehat{\otimes}_{\pi}X_n \to Y$. Thus, if $A: X_1 \times \cdots \times X_n \to Y$ is the continuous $n$-linear operator such that $A^\otimes = T$, then $A$ is not norm-attaining.
\end{proof}

A Banach space $X$ is said to have the {\bf separable complementation property} if every separable subspace of $X$ is contained in a separable complemented subspace. It follows from \cite{amirlind} that every weakly compactly generated Banach space, hence every reflexive Banach space, has the separable complementation property.

\begin{lemma} \label{lema2}
    Let $X_1, \dots, X_n$ be Banach spaces with the separable complementation property, let $Y$ be an arbitrary Banach space, and let $C = B_{{\cal L}_{wsc}(X_1, \dots, X_n; Y)}$. Then, $C^\otimes$ is $SOT$-closed if and only if it is sequentially $SOT$-closed.
\end{lemma}

\begin{proof} 
 The forward implication is trivial. For the converse, suppose that $C^\otimes$ is sequentially $SOT$-closed but not $SOT$-closed. Then, there exists an operator $T \in \overline{C^\otimes}^{SOT} \setminus C^\otimes$. Let $A: X_1 \times \cdots \times X_n \to Y$ be the $n$-linear continuous operator such that $A^\otimes = T$. 
 Since $A \notin C$, it is not weakly sequentially continuous. Consequently, there are weakly convergent sequences $(x_{j,k})_k \subset X_j$, $j=1,\ldots,n$, such that the sequence $(A(x_{1,k}, \dots, x_{n,k}))_k$ is not convergent in $Y$. By passing to a subsequence if necessary, we get that there is $\varepsilon > 0$ such that $\norma{A(x_{1,k}, \dots, x_{n,k})} \geq \varepsilon$ for every $k \in \N$. For each $j = 1, \dots, n$, the set $E_j := \overline{[x_{j,k} : k \in \N]}$ is a separable subspace of $X_j$, and by the separable complementation property of $X_j$, there exists a separable complemented subspace $Z_j$ of $X_j$ containing $E_j$. Thus, the restriction $A|_{Z_1 \times \cdots \times Z_n}$  is not weakly sequentially continuous (as it fails on the generating sequences).

We claim that $(A|_{Z_1 \times \cdots \times Z_n})^\otimes$ belongs to the $SOT$-sequential closure of $B_{{\cal L}_{wsc}(Z_1, \dots, Z_n; Y)}$. For the convenience of the reader, we will check it for $n=2$. Since $Z_1$ and $Z_2$ are separable, we can write $Z_1 = \overline{[z_{1,k}: k \in \N]}$ and $Z_2 = \overline{[z_{2,k}: k \in \N]}$ with $\norma{z_{i,k}} = 1$ for all $i = 1, 2$ and $k \in \N$.
    Because $A^\otimes \in \overline{C^\otimes}^{SOT}$, there exists a net $(B_\lambda)_{\lambda \in \Lambda} \subset C$ such that $$\norma{B_\lambda(x_1,x_2) - A(x_1, x_2)} \longrightarrow 0$$ for all $x_1 \in X_1$ and $x_2 \in X_2$. Restricting each $B_\lambda$ to $Z_1 \times Z_2$, we obtain a net $(B_\lambda|_{Z_1 \times Z_2})_{\lambda \in \Lambda}$ in 
$B_{\mathcal{L}_{wsc}(Z_1, Z_2; Y)}$ that converges pointwise to $A|_{Z_1 \times Z_2}$. We now extract a sequence from this net that still converges pointwise to $A|_{Z_1 \times Z_2}$.
To do so, we construct inductively an increasing sequence $(\lambda_k)_{k \in \mathbb{N}} \subset \Lambda$ as follows:

    $\bullet$ From the convergence $\displaystyle \norma{B_{\lambda}(z_{1,1}, z_{2,1}) - A(z_{1,1}, z_{2,1})} \longrightarrow 0$, there exists $\lambda_1\in \Lambda$ such that $$\norma{B_{\lambda}(z_{1,1}, z_{2,1}) - A(z_{1,1}, z_{2,1})} < 1$$ for every $\lambda \geq \lambda_1$.

    $\bullet$ For the next step, note that for each pair $(i,j)$ with $i,j = 1,2$, we have
$$\lim_{\lambda} \|B_{\lambda}(z_{1,i}, z_{2,j}) - A(z_{1,i}, z_{2,j})\| = 0.$$
Hence, there exists $\mu_{ij} \in \Lambda$ such that
$$\|B_{\lambda}(z_{1,i}, z_{2,j}) - A(z_{1,i}, z_{2,j})\| < \frac{1}{2}
\quad\text{for every } \lambda \geq \mu_{ij}.$$
Since $\Lambda$ is directed, we can choose $\lambda_2 \in \Lambda$ with 
$\lambda_2 > \lambda_1$ and $\lambda_2 \geq \mu_{ij}$ for all $i,j = 1,2$.
Then, for every $\lambda \geq \lambda_2$ and all $i,j = 1,2$, we have
\[
\|B_{\lambda}(z_{1,i}, z_{2,j}) - A(z_{1,i}, z_{2,j})\| < \frac{1}{2}.
\]

 $\bullet$  Now, for $k = 3$, we use the limits
$$\lim_{\lambda} \|B_{\lambda}(z_{1,i}, z_{2,j}) - A(z_{1,i}, z_{2,j})\| = 0 \quad \text{for } i,j = 1,2,3.$$
For each such pair $(i,j)$, choose $\mu_{ij} \in \Lambda$ such that
$$\|B_{\lambda}(z_{1,i}, z_{2,j}) - A(z_{1,i}, z_{2,j})\| < \frac{1}{3}
\quad\text{for every } \lambda \geq \mu_{ij}.$$
Since $\Lambda$ is directed, we can select $\lambda_3 \in \Lambda$ satisfying 
$\lambda_3 > \lambda_2$ and $\lambda_3 \geq \mu_{ij}$ for all $i,j = 1,2,3$.
Consequently,
$$\|B_{\lambda}(z_{1,i}, z_{2,j}) - A(z_{1,i}, z_{2,j})\| < \frac{1}{3}
\quad\text{for every } \lambda \geq \lambda_3 \text{ and all } i,j = 1,2,3.$$
Proceeding inductively, we obtain an increasing sequence $(\lambda_k)_{k \in \mathbb{N}} \subset \Lambda$ such that for each $k \in \mathbb{N}$,
$$\|B_{\lambda_k}(z_{1,i}, z_{2,j}) - A(z_{1,i}, z_{2,j})\| < \frac{1}{k}
\quad\text{for all } i,j = 1,\dots,k.$$

We now show that the sequence $(B_{\lambda_k})_{k \in \mathbb{N}}$ converges to $A|_{Z_1 \times Z_2}$ pointwise on $Z_1 \times Z_2$. 
Let $x_1 \in Z_1$, $x_2 \in Z_2$ and $\varepsilon > 0$ be given. Since $Z_1 = \overline{[z_{1,k} : k \in \N]}$, there exists $u_1 = \displaystyle\sum_{i=1}^{m_1} a_{1,i} z_{1,i}$ such that $$ \|x_1 - u_1\| < \frac{\varepsilon}{6(1 + \|A\|)(1 + \|x_2\|)}. $$
Now, as $Z_2 = \overline{[z_{2,k} : k \in \N]}$, there is $\displaystyle u_2 = \sum_{j=1}^{m_2} a_{2,j} z_{2,j} $ such that
$$ \|x_2 - u_2\| < \frac{\varepsilon}{6(1 + \|A\|)(1 + \|u_1\|)}. $$
Setting $M := \max\{ |a_{1,i}| |a_{2,j}| : 1 \leq i \leq m_1, 1 \leq j \leq m_2\}$, we choose $k_0 \in \mathbb{N}$ such that $k_0 \geq m_1$, $k_0 \geq m_2$ and 
$ \displaystyle  \frac{m_1 m_2 M}{k_0} < \frac{\varepsilon}{3}. $ Then, for every $k \geq k_0$, we have $k \geq \max\{m_1, m_2\}$, and hence
$$ \|(B_{\lambda_k} - A)(z_{1,i}, z_{2,j})\| < \frac{1}{k} \leq \frac{1}{k_0} \quad \text{for all } 1 \leq i \leq m_1 \,\text{ and }\, 1 \leq j \leq m_2. $$
Now we estimate:
\begin{align*}
    \|B_{\lambda_k}(x_1, x_2) - A(x_1, x_2)\|
&\leq \|B_{\lambda_k}(x_1, x_2) - B_{\lambda_k}(u_1, u_2)\| \\
&\quad + \|(B_{\lambda_k} - A)(u_1, u_2)\| \\
&\quad+ \|A(u_1, u_2) - A(x_1, x_2)\|.
\end{align*}
For the first term, using bilinearity and $\|B_{\lambda_k}\| \leq 1$,
\begin{align*}
    \|B_{\lambda_k}(x_1, x_2) - B_{\lambda_k}(u_1, u_2)\|
&= \|B_{\lambda_k}(x_1 - u_1, x_2) + B_{\lambda_k}(u_1, x_2 - u_2)\| \\
&\leq \|x_1 - u_1\| \|x_2\| + \|u_1\| \|x_2 - u_2\| \\
&< \frac{\varepsilon}{6(1 + \|A\|)(1 + \|x_2\|)} \cdot \|x_2\| 
   + \|u_1\| \cdot \frac{\varepsilon}{6(1 + \|A\|)(1 + \|u_1\|)} \\
&< \frac{\varepsilon}{6(1 + \|A\|)} + \frac{\varepsilon}{6(1 + \|A\|)} 
   = \frac{\varepsilon}{3(1 + \|A\|)} < \frac{\varepsilon}{3}.
\end{align*}
For the third term, using $\|A\|$ and the same estimates used above,
\begin{align*}
    \|A(u_1, u_2) - A(x_1, x_2)\|
&= \|A(x_1 - u_1, x_2) + A(u_1, x_2 - u_2)\| \\
&\leq \|A\| \big( \|x_1 - u_1\| \|x_2\| + \|u_1\| \|x_2 - u_2\| \big) \\
&< \|A\| \cdot \frac{\varepsilon}{3(1 + \|A\|)} 
   < \frac{\varepsilon}{3}.
\end{align*}
For the middle term,
\begin{align*}
    \|(B_{\lambda_k} - A)(u_1, u_2)\| 
&\leq \sum_{i=1}^{m_1} \sum_{j=1}^{m_2} |a_{1,i}| |a_{2,j}| \|(B_{\lambda_k} - A)(z_{1,i}, z_{2,j})\| \\
&< \sum_{i=1}^{m_1} \sum_{j=1}^{m_2} M \cdot \frac{1}{k_0} 
   = \frac{m_1 m_2 M}{k_0} < \frac{\varepsilon}{3}.
\end{align*}
Putting the three estimates together, we obtain
$$ \|B_{\lambda_k}(x_1, x_2) - A(x_1, x_2)\| < \frac{\varepsilon}{3} + \frac{\varepsilon}{3} + \frac{\varepsilon}{3} = \varepsilon $$
for all $k \geq k_0$. As $x_1, x_2$ and $\varepsilon > 0$ were arbitrary, this proves that
$(B_{\lambda_k}|_{Z_1 \times Z_2})_k$ converges pointwise to 
$A|_{Z_1 \times Z_2}$.

Now, for each $j = 1, \dots, n$, let $P_j: X_j \to Z_j$ be a bounded linear projection onto $Z_j$ (which exists because each $Z_j$ is complemented in $X_j$). 
Define the sequence $(T_k)_{k}$ by
$$ 
T_k(x_1, \dots, x_n) := \frac{1}{\norma{P_1} \cdots \norma{P_n}} B_{\lambda_k}(P_1(x_1), \dots, P_n(x_n)), \,\, (x_1, \dots, x_n) \in X_1 \times \cdots \times X_n.
$$
It is easy to see that $(T_k)_k$ belongs to $C = B_{\mathcal{L}_{wsc}(X_1, \dots, X_n; Y)}$ and converges pointwise to $\widetilde{A}:=\dfrac{1}{\norma{P_1} \cdots \norma{P_n}} A \circ (P_1, \dots, P_n)$ on $X_1 \times \cdots \times X_n$. 
Moreover, as  $A|_{Z_1 \times \cdots \times Z_n}$ is not weakly sequentially continuous,  $\widetilde{A} \notin C$,
which contradicts the hypothesis that $C^\otimes$ is sequentially $SOT$-closed. 
Hence, $C^\otimes$ must be $SOT$-closed.
\end{proof}

At this point it is natural to ask whether Lemma \ref{lema2} could be deduced from the linear case by means of the linearization of multilinear operators. However, this approach does not work in general. Indeed, the linearization 
$T^\otimes : X_1 \widehat{\otimes}_\pi \cdots \widehat{\otimes}_\pi X_n \to Y $
of a weakly sequentially continuous multilinear operator $T: X_1 \times \cdots \times X_n \to Y$
need not be completely continuous. Indeed, let $X_1$ and $X_2$ be Banach spaces with the so-called Dunford–Pettis property such that $X_1 \widehat{\otimes}_\pi X_2$ fails to have the Dunford–Pettis property (see, for instance, \cite{gonguti, tala} for examples). Then there exists a weakly compact operator $T: X_1 \widehat{\otimes}_\pi X_2 \to c_0$
which is not completely continuous. Let $A: X_1 \times X_2 \to c_0$ be the bilinear operator such that $A^\otimes = T$. Thus, $A$ is weakly compact (see \cite[Proposition 2.12]{multipoli}), and by \cite[Theorem 3.7]{guti} every weakly compact bilinear operator from $X_1 \times X_2$ into $c_0$ is completely continuous. That is, $A$ maps weak Cauchy sequences $(x_{j,k})_k$ in $X_j$ into norm convergent sequences $(A(x_{1,k},x_{2,k}))_k$ in $c_0$, and hence $A$ is weakly sequentially continuous. Nevertheless, its linearization $A^\otimes = T$ fails to be completely continuous. This shows that weak sequential continuity of multilinear operators does not pass, in general, to their linearizations.
Therefore, Lemma \ref{lema2} cannot be obtained directly from the corresponding linear result via linearization. We need one more ingredient before proving Theorem \ref{maintheo}:

\begin{lemma}\label{lemmafin}
     Let $X_1, \dots, X_n$ be Banach spaces such that all continuous $n$-linear operators $X_1 \times \cdots \times X_n$ into $\K$ are weakly sequentially continuous. Then, for any Banach space $Y$, every compact $n$-linear operator from $X_1 \times \cdots \times X_n$ into $Y$ is weakly sequentially continuous. 
\end{lemma}
\begin{proof}
The proof of this lemma will be an adaptation of \cite[Lemma 2.1]{bu2015} to our setting. We begin noticing that, since all continuous $n$-linear operators $X_1 \times \cdots \times X_n$ into $\K$ are weakly sequentially continuous, the mapping $\otimes_{n}\colon X_1 \times \cdots \times X_n\longrightarrow X_1 \widehat{\otimes}_{\pi} \cdots \widehat{\otimes}_{\pi} X_n$ given by $\otimes_{n}(x_{1},\ldots,x_{n})=x_{1}\otimes\cdots\otimes x_{n}$ satisfies
$$x_{1,k} \otimes \cdots \otimes x_{n,k} = \otimes_n(x_{1,k}, \dots, x_{n,k}) \cvf \otimes_n(x_1, \dots, x_n) = x_1 \otimes \cdots \otimes x_n$$
for all sequences $(x_{j,k})_k \subset X_j$ with $x_{j,k} \cvf x_j$ for every $j = 1, \dots, n$.

Now, let $A\colon X_1 \times \cdots \times X_n\longrightarrow Y$ be a compact $n$-linear operator and let, for each $j=1, \dots, n$, $(x_{j,k})_k$ be a weakly convergent sequence in $X_j$ with weak limit $x_j$. Since its linearization $A^{\otimes}\colon X_1 \widehat{\otimes}_{\pi} \cdots \widehat{\otimes}_{\pi} X_n\longrightarrow Y$ is a compact operator (see \cite[Proposition 2.12]{multipoli}), $A^\otimes$ is a completely continuous linear operator, and so 
$$\lim_k A(x_{1,k}, \dots, x_{n,k}) = \lim_k A^\otimes (x_{1,k} \otimes \cdots \otimes x_{n,k}) = A^\otimes(x_1 \otimes \cdots \otimes x_n) = A(x_1, \dots, x_n). $$
Therefore, $A$ is weakly sequentially continuous.
\end{proof}

Now, we are in the position to prove Theorem \ref{maintheo}.

\medskip

\noindent \textit{ Proof of Theorem \ref{maintheo}.}  (1)$\Rightarrow$(2) Let $A: X_1 \times \cdots \times X_n \to Y$ be a $n$-linear continuous operator. Let $(x_{1,k})_k \subset B_{X_1}, \dots, (x_{n,k})_k \subset B_{X_n}$ such that $\displaystyle \lim_{k \to \infty} \norma{A(x_{1,k}, \dots, x_{n,k})} = \norma{A}$. Since $X_1$ is reflexive, there exists a subsequence of $(x_{1,k})_k$ converging to some $x_1 \in B_{X_1}$ in the weak topology. For convenience, we will denote this subsequence by $(x_{1,k})_k$. Since $X_2, \dots, X_n$ are also reflexive, we can assume, by passing to a subsequence if necessary, that each $(x_{j,k})_k$ converges to some $x_j \in B_{X_j}$ in the weak topology. Now, since $A$ is weakly sequentially continuous, $A(x_{1,k}, \dots, x_{n,k}) \to A(x_1, \dots, x_n)$ in $F$. Hence
$$\norma{A(x_1, \dots, x_n)} =  \lim_{k \to \infty} \norma{A(x_{1,k}, \dots, x_{n,k})}  = \norma{A},$$ 
proving that $A$ is norm-attaining.

(2)$\Rightarrow$(1) Let  $A: X_1 \times \cdots \times X_n \longrightarrow Y$ be an $n$-linear continuous operator with $\norma{A} \leq 1$
and let $\lambda = \lambda(A) \geq 1$ be the constant given in our assumption. We claim that $\dfrac{1}{\lambda}A^\otimes$ belongs to $\overline{C^\otimes}^{SOT}$, where $C := B_{{\cal L}_{wsc}(X_1, \dots, X_n; Y)}$. Indeed, let $N(\frac{1}{\lambda}A^\otimes; F, \varepsilon)$ be a basic neighborhood of $\dfrac{1}{\lambda}A^\otimes$. Since $F = \{ z_1, \dots, z_k \} \subset X_1 \widehat{\otimes}_{\pi} \cdots \widehat{\otimes}_{\pi} X_n$, we have from the density of $X_1 \otimes \cdots \otimes X_n$ in $X_1 \widehat{\otimes}_{\pi} \cdots \widehat{\otimes}_{\pi} X_n$, the existence of $u_l = \displaystyle \sum_{j=1}^{m_l} x_{1,j}^l \otimes \cdots \otimes x_{n,j}^l$ such that
$$ \norma{z_l - u_l}_{\pi} < \frac{\lambda \,\varepsilon}{2(\norma{A^\otimes} + \lambda)} $$
for every $l = 1, \dots, k$. Now, for each $i =1, \dots, n$,
$$ K_i := \conj{x_{i,j}^l}{j=1, \dots, m_l, \, l=1, \dots, k}  $$ 
defines a compact subset of $X_i$, so our assumptions yields the existence 
 of an operator $B \in {\cal F}(X_1, \dots, X_n; Y)$ with $\norma{B} \leq \lambda$ such that 
$$ \norma{A(w_1, \dots, w_n) - B(w_1, \dots, w_n)} < \frac{\lambda \, \varepsilon }{2\displaystyle \max_{1 \leq l \leq k} m_l} $$
for all $w_1 \in K_1, \dots, w_n \in K_n$. Thus
\begin{align*}
    \Big\|\dfrac{1}{\lambda}A^\otimes(z_l) - \dfrac{1}{\lambda}&B^\otimes(z_l)\Big\|  \leq \dfrac{1}{\lambda} \left [\norma{A^\otimes (z_l) - A^\otimes (u_l)} + \norma{A^\otimes(u_l) - B^\otimes(u_l)} + \norma{B^\otimes(u_l) - B^\otimes(z_l)}\right ] \\
    & \leq \dfrac{1}{\lambda} [\norma{A^\otimes} + \norma{B^\otimes}] \norma{z_l - u_l}_{\pi} + \dfrac{1}{\lambda} \sum_{j=1}^{m_l} \norma{A(x_{1,j}^l, \dots, x_{n,j}^l) - B(x_{1,j}^l, \dots, x_{n,j}^l)} \\
    &  < \frac{\varepsilon}{2} + \frac{\varepsilon}{2} = \varepsilon,
\end{align*}
which proves that $\dfrac{1}{\lambda}B^{\otimes} \in N(\frac{1}{\lambda}A^\otimes; F, \varepsilon)$. Since finite type $n$-linear operators are weakly sequentially continuous, we conclude that $\dfrac{1}{\lambda} A^\otimes \in \overline{C^\otimes}^{SOT}$.  By Lemma \ref{lema1} and \ref{lema2}, we obtain that $C^\otimes$ is $SOT$-closed, hence $A$ is weakly sequentially continuous. 


(1)$\Rightarrow$(3) If there is a non-weakly sequentially continuous $n$-linear form $\varphi: X_1 \times \cdots \times X_n \to \K$, then for each $y \in S_Y$, the continuous $n$-linear operator $$B_{y}(x_1, \dots, x_n) = \varphi(x_1, \dots, x_n)y, \, x_1 \in X_1, \dots, x_n \in X_n,$$
would fail to be weakly sequentially continuous, contradicting (1). 

Now, let $A \in \mathcal{L}(X_1, \dots, X_n; Y)$ and let $(x_{j,k})_k \subset B_{X_j}$ for each $j = 1, \dots, n$. Since $X_1, \dots, X_n$ are reflexive, there exists a strictly increasing sequence $(k_i)_i \subset \N$ and there are $x_1 \in B_{X_1}, \dots, x_n \in B_{X_n}$ such that $x_{j, k_i} \cvf x_j$ in $X_j$ for every $j=1, \dots, n$. Since $A$ is weakly sequentially continuous by assumption, we get that $A(x_{1,{k_i}}, \dots, x_{n,{k_i}}) \to A(x_1, \dots, x_n)$, proving that $A$ is compact. 

Finally, implication (3)$\Rightarrow$(1) follows from Lemma \ref{lemmafin}. \qed

\medskip

To provide illustrating examples from Theorem \ref{maintheo}, we will need the following proposition. To this end, we recall that a Banach space $X$ is said to have the {\bf $\lambda$-approximation property} ($\lambda$-AP, in short), with $\lambda \geq 1$, if for every compact subset $K \subset X$ and every $\varepsilon > 0$, there exists $T \in \mathcal{F}(X; X)$ with $\norma{T} \leq \lambda$ such that
$ \norma{Tx - x} < \varepsilon$ for every $x\in K.$ If $X$ has the $\lambda$-AP for some $\lambda$, we say that $X$ has the \textbf{bounded approximation property} (BAP, in short).

\begin{proposition} \label{propap}
    Let $X_1, \dots, X_n$ and $Y$ be Banach spaces. Suppose that one of the following holds:\\
    {\rm (1)} $X_1, \dots, X_n$ have the BAP. \\
    {\rm (2)} $Y$ and at least $n-1$ among $X_1, \dots, X_n$ have the BAP. \\
     Then, for each $A \in \mathcal{L}(X_1, \dots, X_n; Y)$ there exists $\lambda = \lambda(A) \geq 1$ such that  for all compact sets $K_1 \subset X_1, \dots, K_n \subset X_n$ and every $\varepsilon > 0$ there exists $B \in \mathcal{F}(X_1, \dots, X_n; Y)$ with $\norma{B} \leq \lambda$ such that
    $$ \norma{A(x_1, \dots, x_n) - B(x_1, \dots, x_n)} < \varepsilon  $$
    for all $x_1 \in K_1,\dots x_n \in K_n$.
\end{proposition}

\begin{proof} Let  $K_1 \subset X_1, \dots, K_n \subset X_n$ be compact sets, $\varepsilon > 0$ and $A \in \mathcal{L}(X_1, \dots,  X_n; Y)$. If $A = 0$, we can simply take $B = 0$. Thus, we assume $A \neq 0$.

    (1) We prove for $n=2$. Supposing that $X_1$ and $X_2$ have the BAP, there are $\lambda_1 \geq 1$ and $\lambda_2 \geq 1$ such that $X_1$ has the $\lambda_1$-AP and $X_2$ has the $\lambda_2$-AP. We let $\lambda = \lambda(A) := \norma{A} \lambda_1 \lambda_2 + 1$.
    Since $X_1$ has the $\lambda_1$-AP, there exists $T_1 \in \mathcal{F}(X_1; X_1)$ with $\norma{T_1} \leq \lambda_1$ such that
    $$ \norma{T_1(x_1) - x_1} < \dfrac{\varepsilon}{2\norma{A} \displaystyle \max\{1, \sup_{x_2 \in K_2} \norma{x_2}\}} $$ for every $x_1 \in K_1$. On the other hand, as $X_2$ has the $\lambda_2$-AP, there exists $T_2 \in \mathcal{F}(X_2; X_2)$ with $\norma{T_2} \leq \lambda_2$  such that
    $ \displaystyle \norma{T_2(x_2) - x_2} < \dfrac{\varepsilon}{2\norma{A} \displaystyle \max\{ 1, \sup_{x_1 \in K_1} \norma{T_1(x_1)}\}}$ for every $x_2 \in K_2$.
    Letting $B(x_1, x_2) := A(T_1(x_1), T_2(x_2))$, we get that $B$ defines a bilinear operator from $X_1\times X_2$ into $Y$ and $ \norma{B} \leq \norma{A} \norma{T_1} \norma{T_2} \leq \lambda $. To check that $B$ has finite type, write $T_i = \displaystyle \sum_{k=1}^{m_i} \varphi_{i,k} \otimes x_{i,k} $ with $\varphi_{i,k} \in X_i^*$ and $x_{i,k} \in X_i$. Thus,
    \begin{align*}
        B(x_1, x_2) & = A\Big(\sum_{k_{1}=1}^{m_1} \varphi_{1,k_1}(x_1)\, x_{1,k_1}, \sum_{k_2=1}^{m_2} \varphi_{2,k_2} (x_2) \,  x_{2,k_2}\Big) \\
        & = \sum_{k_{1} = 1}^{m_1} \sum_{k_2 = 1}^{m_2} \varphi_{1,k_1}(x_1) \varphi_{2,k_2}(x_2) A(x_{1,k_1}, x_{2, k_2}),
    \end{align*}
    proving that $B$ has finite type. Now, given $x_1 \in K_1$ and $x_2 \in K_2$, we get
    \begin{align*}
        \norma{B(x_1, x_2) - A(x_1&, x_2)}  = \norma{A(T_1(x_1), T_2(x_2)) - A(x_1, x_2)} \\
        & \leq \norma{A(T_1(x_1), T_2(x_2)) - A(T_1(x_1), x_2)} + \norma{A(T_1(x_1), x_2) - A(x_1, x_2)} \\
        & = \norma{A(T_1(x_1), T_2(x_2) - x_2)} + \norma{A(T_1(x_1) - x_1, x_2)} \\
        & \leq \norma{A} \norma{T_1(x_1)} \norma{T_2(x_2) - x_2} + \norma{A} \norma{T_1(x_1) - x_1} \norma{x_2} < \varepsilon.
    \end{align*}

(2) We prove for $n=2$. Supposing that $X_1$ and $Y$ have the BAP, there are $\lambda_{1} \geq 1$ and $\lambda_0 \geq 1$ such that $X_1$ has the $\lambda_1$-AP and $Y$ has the $\lambda_0$-AP. Let $\lambda := \lambda_0\lambda_1\norma{A} + 1$.
Since $K_1$ is compact and $X_1$ has the $\lambda_1$-AP, there exists $T_1 \in \mathcal{F}(X_1;X_1)$ such that $\norma{T_1} \leq \lambda_1$ and
$$ \norma{T_1(x_1) - x_1} < \frac{\varepsilon}{2 \norma{A} \displaystyle \max \{ 1, \sup_{x_2 \in K_2} \norma{x_2}\}}  $$
for every $x_1 \in K_1$.  On the other hand, since $Y$ has the $\lambda_0$-AP and $K: = A(T_1(K_1) \times K_2)$ is a relatively compact subset of $Y$, there exists $R \in \mathcal{F}(Y; Y)$ such that $\norma{R} \leq \lambda_0$ and
$$ \norma{Ry - y} < \frac{\varepsilon}{2} $$
for every $y \in K$. Now define $B(x_1, x_2) := R(A(T_1(x_1), x_2))$ for all $x_1 \in X_1$ and $x_2 \in X_2$. Thus, $B$ is a continuous bilinear operator such that $\norma{B} \leq \norma{R} \norma{A} \norma{T_1} \leq \lambda_0 \lambda_1 \norma{A} \leq \lambda$ and
\begin{align*}
    \norma{B(x_1, x_2) - A(x_1, x_2)} & \leq \norma{R(A(T_1(x_1), x_2)) - A(T_1(x_1), x_2)} + \norma{A(T_1(x_1), x_2) - A(x_1, x_2)} \\
    & <\frac{\varepsilon}{2} + \norma{A} \norma{T_1(x_1) - x_1} \norma{x_2} < \varepsilon
\end{align*}
for all $x_1 \in K_1$ and $x_2 \in K_2$. It remains us to check that $B$ has finite type. Indeed, writing $T_1 = \displaystyle \sum_{k=1}^m \varphi_k \otimes u_k$ and $R = \displaystyle \sum_{j=1}^r y_j^\ast \otimes v_j$ with $\varphi_k \in X_1^*$, $u_k \in X_1$, $y_j^* \in Y^*$, and $v_j \in Y$, we have
\begin{align*}
    B(x_1, x_2) & = \sum_{j=1}^r y_j^*(A(T_1(x_1), x_2))v_j =   \sum_{j=1}^r y_j^*(A(\sum_{k=1}^m \varphi_k (x_1) u_k, x_2))v_j  \\ &
    = \sum_{j=1}^r \sum_{k=1}^m \varphi_k(x_1)y_j^*(A(u_k, x_2)) v_j 
\end{align*}
for all $x_1 \in X_1$ and $x_2 \in X_2$, proving that $B \in \mathcal{F}(X_1, X_2; Y)$.

\end{proof}

\begin{remark}
    The conclusion of the preceding proposition may fail even when the range space has the BAP. Indeed, let $X$ be a reflexive Banach space without the BAP and consider the bilinear form $A: X \times X^* \to \K$ given by $A(x,x^*) = x^*(x)$. Suppose, for the sake of contradiction, that there exists $\lambda = \lambda(A) \geq 1$ such that,  for all compact sets $K_1 \subset X, K_2 \subset X^*$ and every $\varepsilon > 0$ there exists $B \in \mathcal{F}(X, X^*; \K)$ with $\norma{B} \leq \lambda$ such that
    $$ |A(x, x^*) - B(x, x^*)| < \varepsilon  $$
    for all $x \in K_1,x^* \in K_2$. Since $X$ is reflexive, $B$ can be written as 
    $$ B(x, x^*) = x^*(S_B(x)), \quad x \in X, x^* \in X^*, $$
where $S_B: X \to X$ is a finite-rank operator. Moreover, $\norma{S_B} = \norma{B} \leq \lambda$ and 
$$ |x^*(x - S_{B}x)| = |x^*(x) - x^*(S_{B}x)| < \varepsilon $$
    for all $x \in K_1$ and $x^* \in K_2$. Hence
    $ {\rm id_X} \in \overline{\conj{S \in \mathcal{F}(X; X)}{ \norma{S} \leq \lambda}}^{\rm WOT}. $
Since the weak and strong operator closures of a convex subset of $\mathcal{L}(X; X)$ coincide, we obtain 
$$ {\rm id_X} \in \overline{\conj{S \in \mathcal{F}(X; X)}{ \norma{S} \leq \lambda}}^{\rm SOT}, $$
which proves that $X$ has the $\lambda$-AP, a contradiction.
\end{remark}

Using Proposition \ref{propap}, the following examples arise from Theorem \ref{maintheo}. 

\begin{examples} \rm \label{ex1} 
    (1) Let $1< p_1, \dots, p_n, q < \infty$ be given. As $\ell_{p_1}, \dots, \ell_{p_n}$ are reflexive Banach spaces with the BAP, we conclude from Theorem \ref{maintheo} and \cite[Theorem 4.2]{alencarfloret} that every $n$-linear continuous operator $A: \ell_{p_1} \times \cdots \times \ell_{p_n} \to \ell_q$ attains its norm if and only if $\displaystyle \sum_{k=1}^n \frac{1}{p_k} < \frac{1}{q}$.

    \noindent  (2) More examples can be found by applying results from Alencar and Floret \cite{alencarfloret, alencarfloret2}. Let us list a few of them:

    (i) Since every bilinear continuous operator from $\ell_5 \times \ell_5$ into $L_p([0,1]), \, 1 < p \leq 2$, is weakly sequentially continuous (see \cite[5.5]{alencarfloret}), we get that all bilinear continuous operators from $\ell_5 \times \ell_5$ into $L_p([0,1]), \, 1 < p \leq 2$, are norm-attaining. 

    (ii) The Tsirelson’s original space $T^{\ast}$ and its dual $T$ are reflexive spaces with Schauder bases. If $1 < p \leq 2$, then $L_{p}(\mu)$ has type $p$ (see \cite[Theorem 6.2.14]{albiac}). Thus, it follows by  \cite[Theorem 2.3]{alencarfloret2} and \cite[Corollary 5.6]{alencarfloret} that every $n$-linear continuous operator $A\colon (T^{\ast})^{n}\longrightarrow L_{p}(\mu)$ is weakly sequentially continuous whenever $1 < p \leq 2$, hence the Theorem \ref{maintheo} guarantees that $A$ attains its norm.

    (iii) Let $1 < p_1, \dots, p_n < \infty$ such that $\displaystyle \sum_{k=1}^n \frac{1}{p_k} < \frac{1}{2}$. Thus, if $Y$ is a weak Hilbert space (see \cite[Definition 12.1]{pisier}), it follows from a combination of Proposition 5.7 and Corollary 1 of \cite{alencarfloret} that every $n$-linear continuous operator from $\ell_{p_1} \times \cdots \times \ell_{p_n}$ into $Y$ is weakly sequentially continuous, and by Theorem \ref{maintheo} all operators from $\ell_{p_1} \times \cdots \times \ell_{p_n}$ into $Y$ are norm-attaining. For examples of weak Hilbert spaces, we refer \cite[Section 13]{pisier}, the twisted Hilbert space introduced in \cite{fuente}, and \cite{pino}.

\noindent (3) Theorem \ref{maintheo} fails if at least one of the spaces $X_i$ is not reflexive. If $X$ is a non reflexive Banach space, then there exists a non-norm attaining continuous linear functional $\varphi \in X^*$. Thus, the bilinear operator $A: X \times \K^n \to \K$ defined by $A(x,(a_1, \dots, a_n)) = \varphi(x)a_1$ is not norm-attaining. However, every bilinear operator from $X \times \K^n$ in $\K$ is weakly sequentially continuous.
\end{examples}

We now present Banach spaces $X_1, \dots, X_n$ and $Y$ such that $Y$ and at least one of the spaces $X_i$ fail the BAP, while the conclusion of Proposition \ref{propap} still holds. Thus, the converse of Proposition \ref{propap} fails in general. For this, we recall from \cite{maite} that a continuous $n$-linear operator $A: X_1 \times \cdots \times X_n \to Y$ is said to {\bf factor through a Hilbert space} if there exists a Hilbert space $H$, a subset $M \subset H$, a continuous $n$-linear operator $B: X_1 \times \cdots \times X_n \to H$ whose image is contained in $M$, and a Lipschitz map $f: M \to Y$ such that the following diagram commutes:
\[
\begin{tikzcd}
X_1 \times \cdots \times X_n \arrow[r, "A"] \arrow[d, "B"'] & Y \\
M \arrow[ur, "f"'] \arrow[d, hook] & \\
H &,
\end{tikzcd}
\]
where $M \hookrightarrow H$ denotes the inclusion map. In particular, one could replace the class of Hilbert spaces for a different class, such as the class of Banach spaces with the approximation property. In our next result, we will require the factorization through a Banach space with the BAP, but the map $f$ will be required to be a linear operator defined in the whole space.

\begin{proposition} \label{propapfact}
    Let $X_1, \dots, X_n$ and $Y$ be Banach spaces such that every continuous $n$-linear operator $A: X_1 \times \cdots \times X_n \to Y$ factors through a Banach space with the BAP, that is there exist a Banach space $Z = Z(A)$ with the BAP, a continuous $n$-linear operator $B: X_1 \times \cdots \times X_n \to Z$, and a continuous linear operator $S: Z \to Y$ such that $A = S \circ B$. Assume that at least $n-1$ among $X_1, \dots, X_n$ have the BAP.
    Then, for every $A \in \mathcal{L}(X_1, \dots, X_n; Y)$, there exists $\lambda = \lambda(A) \geq 1$ such that for all compact sets $K_1 \subset X_1, \dots, K_n \subset X_n$ and for every $\varepsilon > 0$, there exists $B \in \mathcal{F}(X_1, \dots, X_n; Y)$ with $\norma{B} \leq \lambda$ and
    $$ \norma{A(x_1, \dots, x_n) - B(x_1, \dots, x_n)} < \varepsilon  $$
    for all $x_1 \in K_1,\dots x_n \in K_n$.
\end{proposition}

\begin{proof}
Let  $K_1 \subset X_1, \dots, K_n \subset X_n$ be compact sets, $\varepsilon > 0$ and $A \in \mathcal{L}(X_1, \dots,  X_n; Y)$.
By the assumption, there exist a Banach space $Z$ with the BAP, a continuous $n$-linear operator $B: X_1 \times \cdots \times X_n \to Z$, and continuous linear operator $S: Z \to Y$ such that $A = S \circ B$.
Without loss of generality, we may assume that $S \neq 0$.
Now, since $Z$ has the BAP, item (2) of Proposition \ref{propap} ensures the existence of  $\lambda = \lambda(A)$ and
$\widetilde{B} \in \mathcal{F}(X_1, \dots, X_n; Z)$ 
with $\norma{\widetilde{B}} \leq \lambda$
such that
 $$ \norma{B(x_1, \dots, x_n) - \widetilde{B}(x_1, \dots, x_n)} < \frac{\varepsilon}{\norma{S}} $$
for all $x_1 \in K_1, \dots, x_n \in K_n$. Thus, $S \circ \widetilde{B} \in \mathcal{F}(X_1, \dots, X_n; Y)$ satisfies
 \begin{align*}
    \norma{A(x_1, \dots, x_n) - S\circ \widetilde{B}(x_1, \dots, x_n)} & = \norma{S \circ B(x_1, \dots, x_n) - S \circ \widetilde{B} (x_1, \dots, x_n)} \\
    & \leq \norma{S} \norma{B(x_1, \dots, x_n) - \widetilde{B}(x_1, \dots, x_n)} < \varepsilon
\end{align*} 
for all $x_1 \in K_1, \dots, x_n \in K_n$.
\end{proof}

\begin{example} \rm Let $1<q<2<p<\infty$. By \cite{skan}, we can find closed subspaces $Y$ of $\ell_{q}$ and $X$ of $\ell_{p}$ without BAP. Since 
$\ell_p$ has type $2$ and $\ell_q$ has cotype $2$ (see \cite[Theorem 6.2.14]{albiac}), we get that $X$ has type $2$ and $Y$ has cotype $2$. By \cite[Corollary 3.5(i)]{bu2026}, $X \widehat{\otimes}_{\pi} \mathbb{K}^n$ has type $2$, and so every continuous linear operator from $X \widehat{\otimes}_{\pi} \K^n$ into $Y$ factors thorough a Hilbert space (see \cite[Theorem 7.4.2]{albiac}). Thus, if $A: X \times \K^n \to Y$ is a continuous bilinear operator, there exists a Hilbert space $H$ such that the following diagram commutes
\[
\begin{tikzcd}
X \widehat{\otimes}_{\pi} \K^n \arrow[r, "A^\otimes"] \arrow[d, "R"'] & Y \\
H \arrow[ur, "S"']  & 
\end{tikzcd}
\]
From the above discussion, we have that $(X,\mathbb{K}^{n};Y)$ satisfy the thesis in Proposition \ref{propapfact}. Moreover, as an application of Pitt's theorem and the ideal property for compact operators, we get that every bounded linear operator from $X$ into $Y$ is compact. From this, it is straightforward to prove that every continuous  bilinear operator $X \times \K^n \to Y$ is weakly sequentially continuous. 
Thus, Theorem \ref{maintheo} ensures that every continuous bilinear operator from $X \times \K^n$ into $Y$ is norm-attaining.
\end{example}

In the remainder of this section, we investigate possible extensions of Theorem~1.1 to the setting of homogeneous polynomials. First, we say that $P \in \mathcal{P}(^n X; Y)$ attains its norm if exists $x \in B_X$ such that $\norma{P(x)} = \norma{P}$. We notice that $P$ can attain its norm on a vector $x \in S_X$ and $T_P$, its associated symmetric $n$-linear operator, fails to attain its norm \cite{pappas}. This observation motivates the search for a polynomial analogue of Theorem \ref{maintheo}.

To do this, we recall that the symmetric $n$-fold algebraic tensor product $\otimes_{n,s} X$ is the vector subspace of $X \otimes \cdots \otimes X$ generated by $\conj{\theta_n(x)}{x \in X}$, where $\theta_n(x) = x \otimes\cdots \otimes x$. The projective symmetric $n$-fold tensor product $\widehat{\otimes}_{n,s, \pi} X$ is the completion of $\otimes_{n,s} X$ with respect to the symmetric projective tensor norm
$$ \norma{u}_{s, \pi} = \inf \conj{\sum_{i=1}^k \norma{x_i}^n}{u = \sum_{i=1}^k \theta_n(x_i)}. $$
Moreover, for every continuous $n$-homogeneous polynomial $P : X \to Y$, there exists a unique bounded linear operator $P^\otimes: \widehat{\otimes}_{n,s, \pi} X \to Y$ such that $P(x) = P^\otimes(\theta_n(x))$ for every $x \in X$. Furthermore, the mapping
$$ P \in \mathcal{P}(^n X; Y) \mapsto P^\otimes \in \mathcal{L}(\widehat{\otimes}_{n,s, \pi} X; Y) $$
is an isometric isomorphism. 

The same argument used in Remark \ref{remark}, using the polynomial version of the Uniform Boundedness Principle \cite[Theorem 2.6]{mujicalivro}, proves the following:

\begin{remark} \label{remarkpol} 
    Let $X$ and $Y$ be Banach spaces, $(P_\alpha)_\alpha$ be a norm bounded net in $\mathcal{P}(^n X; Y)$, and $P \in \mathcal{P}(^n X; Y)$. Then, $P_\alpha^\otimes \to P^\otimes$ in the $SOT$ of ${\cal L}(\widehat{\otimes}_{n,s, \pi} X; Y)$ if and only if $(P_\alpha)_\alpha$ converges pointwise to $P$.
\end{remark}

For a subset $C$ of ${\cal P}(^n X; Y)$, we define $C^\otimes = \conj{P^\otimes}{P \in C}$. The following, as Lemma \ref{lema1}, is an application of \cite[Theorem A]{sheldon}.

\begin{lemma} \label{lema1pol}
Let $X$ and $Y$ be Banach spaces. Assume that exists a norm-closed convex subset $C \subset \mathcal{P}(^nX; Y)$ such that the set $C^\otimes = \conj{P^\otimes}{P \in C}$ is not sequentially $SOT$-closed in ${\cal L}(\widehat{\otimes}_{n,s,\pi} X; Y)$. Then, there exists a non-norm attaining $n$-homogeneous polynomial $P: X \to Y$.
\end{lemma}

Recall that an $n$-homogeneous polynomial $P: X \to Y$ is {\bf weakly sequentially continuous} if $P(x_{k})\longrightarrow P(x)$ in $Y$ whenever $x_{k}\xrightarrow{ \,\omega} x$ in $X$. The collection of all weakly sequentially continuous polynomials from $X$ into $Y$ is denoted by ${\cal P}_{wsc}(^n X; Y)$. Our following lemma can be seen as a polynomial version of Lemma \ref{lema2}. At this point, it is important to observe that we cannot use linearization techniques as ${\cal P}_{wsc}(^n X; Y) \neq \mathcal{L}_{wsc}(\widehat{\otimes}_{n,s, \pi }X; Y)$ due to \cite{bu2015}.

\begin{lemma} \label{lema2pol}
    Let $X$ be Banach spaces with the separable complementation property, let $Y$ be an arbitrary Banach space, and let $C = B_{{\cal P}_{wsc}(^nX; Y)}$. Then, $C^\otimes$ is $SOT$-closed if and only if it is sequentially $SOT$-closed.
\end{lemma}

\begin{proof}
     The forward implication is trivial. For the converse, suppose that $C^\otimes$ is sequentially $SOT$-closed but not $SOT$-closed in ${\cal L}(\widehat{\otimes}_{n, s, \pi} X; Y)$. Then there exists an operator $T \in \overline{C^\otimes}^{SOT} \setminus C^\otimes$. Let $P \in \mathcal{P}(^n X; Y)$ such that $P^\otimes = T$. Since $P^\otimes \notin C^\otimes$, $P \notin C$, that is $P$ is not weakly sequentially continuous. Consequently, there is a weakly convergent sequence $(x_k)_k \subset X$ such that $(P(x_k))_k$ is not convergent in $Y$. By passing to a subsequence if necessary, we get that $\norma{P(x_k)} \geq \varepsilon$ for every $k \in \N$. Since $E := \overline{[x_k : k \in \N]}$ is a separable subspace of $X$, it follows by the separable complementation property of $X$ that there exists a separable complemented subspace $Z$ of $X$ containing $E$. Thus, the restriction $P|_{Z}$  is not weakly sequentially continuous. The same argument used in the proof of Lemma \ref{lema2} can be adapted to prove that the linearization of $P|_{Z}$  belongs to the $SOT$-sequential closure of $(B_{\mathcal{P}_{wsc}(^n Z; Y))})^\otimes$. Thus, there exists a sequence $(P_k)_k$ in $B_{\mathcal{P}_{wsc}(^n Z; Y))}$ such that
     $$ \lim_{k \to \infty}\norma{P_k(z) - P(z)} = 0 $$
     for every $z \in Z$. If $\pi: X \to Z$ is the bounded linear projection onto $Z$, we can define a sequence $(Q_k)_k$ by
     $$ Q_k(x) := \frac{1}{\norma{\pi}^n} P_{k}(\pi(x)), \quad x \in X. $$
     Thus, $(Q_k)_k$ belongs to $C$ and
     $ \displaystyle \lim_{k \to \infty} \left \|Q_k(x) -  \frac{1}{\norma{\pi}^n}(P\circ\pi)(x)\right \| = 0 $
     for every $x \in X.$ Finally, as $P$ is not weakly sequentially continuous and $\pi$ is a projection, $ \frac{1}{\norma{\pi}^n} P \circ \pi \notin C$, which is a contradiction.
\end{proof}

To present our next result, we recall that an $n$-homogeneous polynomial $P: X \to Y$ is {\bf compact} if $P(B_X)$ is relatively compact subset of $Y$. 
The collection of all such polynomials is denoted by $\mathcal{P}_{\cal K}(^n X; Y)$.

\begin{theorem} \label{maintheopol}
Let $X$ be a reflexive Banach space and let $Y$ be a Banach space.  Assume that for every $P \in \mathcal{P}(^nX; Y)$, there exists $\lambda = \lambda(P) \geq 1$ such that for every compact set $K \subset X$ and  every  $\varepsilon > 0$, there exists $Q \in \mathcal{P}(^nX; Y)$ with $\norma{T_Q} \leq \lambda$ such that $T_Q$ has finite type and
    $ \norma{P(x) - Q(x)} < \varepsilon  $
    for every $x\in K$.
    Then,  the following are equivalent: \\
    {\rm (1)} Every continuous $n$-homogeneous polynomial $P: X \to Y$ is weakly sequentially continuous. \\
    {\rm (2)} Every continuous $n$-homogeneous polynomial $P: X \to Y$  is norm-attaining. \\
   {\rm (3)} Every $P \in \mathcal{P}(^n X; \K)$ is weakly sequentially continuous and every $P \in \mathcal{P}(^n X; Y)$ is compact.
\end{theorem}

\begin{proof}
The implications (1)$\Rightarrow$(2) and (1)$\Rightarrow$(3) follow from the same argument used in Theorem \ref{maintheo}, and implication (3)$\Rightarrow$(1) follows from \cite[Lemma 2.1]{bu2015}. 
The converse implication (2)$\Rightarrow$(1) is obtained by adapting the proof of (2)$\Rightarrow$(1) in Theorem \ref{maintheo}, combined with Lemmas \ref{lema1pol} and \ref{lema2pol}.
\end{proof}

The following result will be useful for constructing examples illustrating Theorem \ref{maintheopol}. Its proof is an adaptation of the proof of Proposition \ref{propap} for the polynomial set. 

\begin{proposition} \label{propappol}
   Let $X$ and $Y$ be Banach spaces. If $X$ has the BAP, then for every $P\in \mathcal{P}(^nX; Y)$, there exists $\lambda = \lambda(P) \geq 1$ such that for every compact set $K \subset X$ and every $\varepsilon > 0$, there exists $Q \in \mathcal{P}(^n X; Y)$ with $\norma{T_Q} \leq \lambda$ such that $T_Q$ has finite type and $ \norma{P(x) - Q(x)} < \varepsilon  $
    for every $x \in K.$
\end{proposition}

\begin{proof}
    Let  $K \subset X$ be compact set, $\varepsilon > 0$ and $0 \neq P \in \mathcal{P}(^nX; Y)$. For convenience, we will check this proposition for $n=2$.

    Suppose first that $X$ has the $\lambda_0$-BAP for some $\lambda_0 \geq 1$. Thus, there exists $S \in \mathcal{F}(X; X)$ with $\norma{S} \leq \lambda_0$ such that $$\norma{S(x) - x} < \dfrac{\varepsilon}{ (\lambda_0 + 1) \norma{T_P}\displaystyle \max\{ 1, \sup_{x \in K} \norma{x} \}}$$
    for every $x \in K$.
    Defining $A : X\times X \to Y$ by 
    $$A(x_1, x_2) = T_P(S(x_1), S(x_2)), \, x_1, x_2 \in X,$$
    we get that $A$ is symmetric - as $T_P$ is symmetric, and $A$ has finite type by the proof of Proposition \ref{propap}.     
    So, if $Q: X \to Y$ is the continuous $2$-homogeneous polynomial such that $T_Q = A$, we get that $T_Q$ has finite type and
    $\norma{T_Q} \leq \norma{T_P} \norma{S}^2$. Letting $\lambda(P):= \norma{T_P} \norma{S}^2 + 1$, we have
    \begin{align*}
        \norma{P(x) - Q(x)} & = \norma{T_P(x, x) - T_P(S(x), S(x))} \\
        & \leq \norma{T_P(x, x) - T_P(S(x), x)} + \norma{T_P(S(x), x) - T(S(x), S(x))} \\
        & = \norma{T_P(x - S(x), x)} + \norma{T_P(S(x), x - S(x))} \\
        & \leq \norma{T_P} \norma{x - S(x)} \norma{x} + \norma{T_P} \norma{x - S(x)} \norma{S} \norma{x} \\
        & = \norma{T_P} \norma{x - S(x)} \sup_{x \in K} \norma{x} (1+ \norma{\lambda_0}) < \varepsilon.
    \end{align*}
\end{proof}

Using Theorem \ref{maintheopol} and Proposition \ref{propappol}, we get the following examples:

\begin{examples}
    \rm (1) Let $n\geq 2$ and $1<p,q<\infty$. It follows from \cite[4.3]{alencarfloret} and Theorem \ref{maintheopol} that every $P\in \mathcal{P}(^n \ell_p; \ell_q)$ attains its norm if and only if  $nq<p$. \\
    \rm (2) By \cite[Theorem 2.3]{alencarfloret2}, we have that every continuous $m$-homogeneous polynomial from $T^\ast$, the Tsirelson`s original space, to $L_{p}(\mu)$ whenever $1<p\leq 2$  is weakly sequentially continuous. Thus, every $P\in \mathcal{P}(^n T^{\ast}; L_{p}(\mu))$ attains its norm. \\
    \rm (3) Let $n \in \N$ be fixed and $p > 2n$. If $Y$ is a weak Hilbert space, we have from \cite{alencarfloret} that every continuous $n$-homogeneous polynomial from $\ell_p$ into $Y$ is weakly sequentially continuous, and by Theorem \ref{maintheopol} we obtain that all $n$-homogeneous polynomials from $\ell_p$ into $Y$ are norm attaining. \\ 
    (4) Consider $P \in \mathcal{P}(^2 \ell_2; \K)$ defined by
    $\displaystyle P(x)=\displaystyle x_{1}^{2}+\sum_{n=2}^{\infty}\dfrac{n}{n+1}x_{n}^{2}$
    for every $x = (x_n)_n \in \ell_2$. Since $P(0) = 0$ and $(e_n)_{n}$ is a weakly null sequence in $\ell_2$ such that $\displaystyle \lim_{n \to \infty} P(e_n) = 1$, 
    $P$ is not weakly sequentially continuous. However, $P$ attains its norm as $\|P\|=1=P(e_{1})$. This does not contradict Theorem \ref{maintheopol}, since there exist non-norm attaining continuous $2$-homogeneous polynomials from $\ell_{2}$ to $\mathbb{K}$ (see, e.g., \cite[p. 7]{arongarciapelteix}).\\
    \rm(5) The reflexivity cannot be dropped in the Theorem \ref{maintheopol}. Indeed, by \cite[p. 1127]{bu2015} every $n$-homogeneous polynomial $P\colon c_0\to \ell_p$, $1\leq p<\infty$, is weakly sequentially continuous. On the other hand, the $2$-homogeneous polynomial $P\colon c_0\to \ell_2$, $\displaystyle P((x_n)_n)=\left(\frac{x_n^2}{n}\right)_n$, does not attain its norm, since $\displaystyle \|P\|=\frac{\pi}{\sqrt{6}}$ but the supremum  is attained at the sequence $(\pm1,\pm1,\dots)$, which does not belong to $c_0$.
\end{examples}

\section{The Banach lattice setting}\label{section3}

 As mentioned in the introduction, in this section we present results analogous to those in Section 2, concerning positive multilinear operators and positive homogeneous polynomials between Banach lattices. In the Banach lattice setting, the following result was proved by the second and third named authors in \cite{luizmiranda}:

\begin{theorem} \label{teolumir}
     Let $E$ be a reflexive Banach lattice and let $F$ be a Dedekind complete Banach lattice. Consider the following conditions: \\
{\rm (1)} Every positive operator $T: E \to F$ is compact. \\
{\rm (2)} Every positive operator $T: E \to F$ attains its norm. \\
{\rm (3)} $B_{\mathcal{K}^+(E; F)}$ is sequentially closed in the absolutely strong operator topology. \\
Then {\rm (1)}$\Rightarrow${\rm(2)}$\Rightarrow${\rm(3)}. In addition, if the order on $E$ is induced by a basis, or if $F$ has an order continuous norm and its order is induced by a basis, then {\rm(3)}$\Rightarrow${\rm(1)}.
\end{theorem}

The {\bf absolutely strong operator topology} ($|SOT|$, for short) mentioned in the above theorem was introduced in \cite{luizmiranda} to be the topology in $\mathcal{L}^r(E; F)$ - notice that we are assuming that $F$ is Dedekind complete - defined by the following basic neighborhoods
$$ N^r(T; A, \varepsilon) = \conj{S \in \mathcal{L}^r(E; F)}{\norma{|T - S|(x)} < \varepsilon, \, \text{for every } x \in A}, $$
where $A \subset E^+$ is an arbitrary finite set and $\varepsilon > 0$. Thus, a net $(T_\alpha)_\alpha \subset \mathcal{L}^r(E; F)$ converges to $T$ in the $|SOT|$ if and only if $(|T_\alpha - T|(x))_\alpha$ converges to $0$ for every $x \in E^+$.

As a consequence of the proof of implications (2)$\Rightarrow$(3)$\Rightarrow$(1) of Theorem \ref{teolumir} presented in \cite{luizmiranda}, we have the following corollary.

\begin{corollary} \label{corret}
    Let $E$ and $F$ be Banach lattices with $F$ being Dedekind complete such that every positive operator from $E$ into $F$ attains its norm. Assume that every $T \in B_{{\cal L}^+(E; F)}$
    is the $|SOT|$-limit of a sequence $(T_n)_n \subset B_{{\cal F}^+(E; F)}$.
    Then, every positive operator from $E$ into $F$ is compact.
\end{corollary}

\begin{proof}
     Let $T: E \to F$ be a positive operator with $\norma{T} \leq 1$. By the assumption, there exists a sequence $(T_n)_n \subset B_{\mathcal{F}^+(E; F)}$ such that $\displaystyle \lim_{n \to \infty} \norma{|T_n - T|(x)} = 0$ for every $x \in E$. However, since every positive operator from $E$ into $F$ attains its norm, we get from \cite[Lemma 2.9]{luizmiranda} that $B_{{\cal K}^+(E; F)}$
    is sequentially closed in the absolutely strong operator topology. Therefore, $T$ is compact.
\end{proof}

Notice that we do not need to assume that $E$ is reflexive in Corollary \ref{corret} since every positive operator define in AM-space with unit is norm attaining (see \cite[Exercise 2, p. 270]{positiveoperators}). This is a striking difference with the non-positive case since for every non-reflexive Banach space $X$ and every Banach space $Y$, there exists a non-norm attaining operator $T: X \to Y$. Additionally, we remark that if $E$ or $F$ satisfies the property defined below, then the pair $(E, F)$ satisfy the approximation property in the hypothesis of Corollary \ref{corret}.

\begin{definition} \label{apret}
    A Banach lattice $E$ is said to have the {\bf sequentially positive approximation property} (SPAP, in short) if there exists a sequence $(S_n)_n \subset B_{{\cal F}^+(E; E)}$ such that 
    $$\displaystyle \lim_{n \to \infty} \norma{|S_n - {\rm id}_{E}|(x)} = 0$$ for every $x \in E^+$, that is $(S_n)_n$ converges to ${\rm id}_{E}$ in the $|SOT|$.
\end{definition}

Observe that, in the definition above, it is implicitly assumed that
the modulus $|S_n-{\rm id}_E|$ exists for every $n\in\mathbb{N}$. In particular, if $E$ is a Dedekind complete Banach lattice, then this modulus always exists.

It will follow from Proposition \ref{propapfact}, in a multilinear sense, that if $E$ or $F$ satisfies the sequentially positive approximation property, then the pair $(E, F)$ satisfy the hypothesis of Corollary \ref{corret}. In the Banach lattice setting, there are several analogues of the approximation property (AP) that take the order structure into account. Among them are the positive approximation property (PAP) \cite{nielsen}, in which the identity operator is approximated by positive finite-rank operators; the regular approximation property (RAP) \cite{bupams}, where the approximation is considered via the modulus of regular operators; and the lattice approximation property (LAP) \cite{blanco}, which coincides with the PAP in Dedekind complete Banach lattices. It is not difficult to check that the sequentially positive approximation property implies the  RAP and the PAP. 
Next, we present a characterization of the SPAP. In particular, we show that this new property is a separable version of the RAP.

\begin{proposition} \label{caracspap}
    For an infinite dimensional Banach lattice $E$, the following are equivalent: \\
    {\rm (1)} $E$ has the $SPAP$. \\
    {\rm (2)} $E$ is separable, atomic, and has order continuous norm. \\
    {\rm (3)} The order of $E$ is given by a Schauder basis $(x_n)_n$ consisting of mutually disjoint atomic elements. \\
    {\rm (4)} $E$ is separable and has the RAP.
\end{proposition}

\begin{proof}
    (1)$\Rightarrow$(2) Assuming that $E$ has the SPAP, there exists a sequence $(S_n)_n \subset B_{\mathcal{F}^+(E; E)}$ such that
    $\displaystyle  \lim_{n \to \infty} \norma{|S_n - {\rm id}_E|(x)} = 0 $
    for every $x \in E^+$. This implies that 
    $$\norma{S_n(x) - {\rm id}_E(x)} \leq \norma{|S_n - {\rm id}|(|x|)} \to 0$$ for every $x \in E^+$, and consequently
    $ E = \displaystyle \overline{\operatorname{span } \left ( \bigcup_{n \in \N} S_n(E) \right )}. $
    Since each $S_n(E)$ is finite dimensional, a straightforward argument shows that $E$ is separable.
    To prove that $E$ is atomic and has order continuous norm, it suffices to check that every order interval in $E$ is compact (see \cite[Theorem 6.1]{WOrder}). Indeed, we will prove that, for each $u \in E^+$, the order interval $[-u,u]$ is totally bounded. Let  $\varepsilon > 0$ be given. Since
    $$  \norma{S_n(x) - {\rm id}_E(x)} \leq  \norma{|S_n - {\rm id}_E|(|x|)} \leq \norma{|S_n - {\rm id}_E|(u)} $$
    holds for all $x \in [-u,u]$ and $n \in \N$, there exists $k_{\varepsilon} \in \N$ such that 
    $$ \sup_{x \in [-u,u]} \norma{S_{k_\varepsilon}(x) - {\rm id}_E(x)} \leq \norma{|S_{k_{\varepsilon}} - {\rm id}_E|(u)} < \frac{\varepsilon}{2}. $$
    Since $S_{k_{\varepsilon}}([-u,u])$ is a norm bounded subset of the finite-dimensional space $S_{k_{\varepsilon}}([-u,u])$, it is totally bounded, so there are $y_1, \dots, y_m \in E$ such that
    $\displaystyle  S_{k_{\varepsilon}}([-u,u]) \subset \bigcup_{i=1}^m B(y_i; \frac{\varepsilon}{2}). $
    Thus, for each $x \in [-u,u]$, there exists $j \in \{1, \dots, m\}$ such that $\norma{S_{k_{\varepsilon}}(x) - y_j} < \dfrac{\varepsilon}{2}$, and so
    $$ \norma{x - y_j} \leq \norma{S_{k_{\varepsilon}}(x) - {\rm id}_E (x)} + \norma{S_{k_{\varepsilon}}(x) - y_j} < \varepsilon.  $$
    This proves that $[-u,u] \subset \displaystyle \bigcup_{i=1}^m B(y_j; \varepsilon)$. Then $[-u,u]$ is totally bounded, hence compact. As every order interval in $E$ can be seen as a subset of an order interval of the form $[-u,u]$, we conclude this implication.

    (2)$\Rightarrow$(3) Let $E$ be a separable and atomic Banach lattice with an order continuous norm. We may assume that $E \neq \{0\}$. By \cite[Corollary 3.5]{kandicvev}, there are at most countably many normalized positive atoms.
    If there are only finitely many normalized positive atoms, say $e_1, \dots, e_m$, the dense band decomposition property gives $E = \operatorname{ span } \{e_1, \dots, e_m\}$, a contradiction. Thus, there are infinitely many normalized positive atoms, and enumerate them as $(e_j)_j$.  By \cite[Remark 4.13]{kandicmar}, each $I_{e_j}$ is a projection band, and the corresponding band projection $P_{e_j}$ has form $P_{e_j}(x) = e_j^*(x) e_j$ for every $x \in E$, with $0 \leq e_j^* \in E^*$. Moreover, since $E$ is atomic and has order continuous norm, these bands form a dense band decomposition of $E$. For each $N \in \N$, set $R_N := \displaystyle \sum_{j=1}^N P_{e_j}$.  Since $e_i \perp e_{k}$ whenever $i \neq k$, $I_{e_i} \perp I_{e_k}$
    whenever $i \neq k$, so $R_N$ is the band projection onto $\operatorname{span} \{ e_1, \dots, e_N\}$. In particular, $0 \leq R_N \leq {\rm id}_{E}$ and $\norma{R_N - {\rm id}_E} \leq 1$. We now prove that $(e_j)_j$ is a Schauder basis for $E$. Fix $x \in E$ and $\varepsilon > 0$. By \cite[Lemma 4.10]{kandicmar}, there is a nonempty finite set $F \subset \N$ such that $\displaystyle \left \| x - \sum_{j \in F} P_{e_j}(x)  \right \| < \varepsilon$.
    For every $N \geq \max F$, we have
    \begin{align*}
        \norma{x - R_N(x)} & \leq \left \|({\rm id}_E - R_N) \Bigg(x - \sum_{j \in F} P_{e_j}(x)\Bigg)\right\| + \left \|({\rm id}_E - R_N) \Bigg(\sum_{j \in F} P_{e_j}(x))\Bigg)\right\| \\
        & \leq \norma{{\rm id}_E - R_N} \left \|x - \sum_{j \in F} P_{e_j}(x)\right\| + 0 < \varepsilon.
    \end{align*}
Consequently, $ x = \displaystyle \lim_{N \to \infty} R_N(x) = \sum_{j=1}^\infty e_j^*(x) e_j$ in norm. To prove uniqueness, suppose that $x = \displaystyle \sum_{j=1}^\infty a_j e_j$ in norm. For each $i \in \N$, we have
$$ e_i^*(x) e_i = P_{e_i}(x) = P_{e_i} (\sum_{j=1}^\infty a_j e_j) = \sum_{j=1}^\infty a_j P_{e_i}(e_j) = a_i e_i,  $$
which implies that $a_i = e_i^*(x)$ for every $i \in \N$. Hence $(e_j)_j$ is a Schauder basis for $E$. It remains us to see that the order of $E$ is given by this basis, that is
$$ x \in E^+ \iff e_j^*(x) \geq 0 \,\,  \text{ for every } j \in \N. $$
Indeed, if $e_j^*(x) \geq 0$ for every $j \in \N$, then $R_N(x) \geq 0$ for every $N \in \N$. As $R_N(x) \to x$ in norm and $E^+$ is closed, we obtain $x \geq 0$. The converse implication follows from the fact that each $e_j^*$ is a positive linear functional. 


    (3)$\Rightarrow$(1) Let $(e_n)_n$ be Schauder basis of $E$ consisting of mutually disjoint atomic vectors that determines the order of $E$, and let $(e_n^*)_n$ be the corresponding coefficient functionals. For each $N \in \N$, define $R_N(x) := \displaystyle \sum_{j=1}^N e_j^*(x) e_j$ for every $x \in E$. Since the order of $E$ is given by this basis, for every $x \in E^+$, we have $e_j^*(x) \geq 0$ for every $j \in \N$, so
    $$ 0 \leq R_N(x) \leq \sum_{j=1}^\infty e_j^*(x) = x. $$
    Thus $\norma{R_N(x)} \leq \norma{x}$ for every $x \in E^+$, which implies that $\norma{R_N} \leq 1$. Therefore $(R_N)_N \subset B_{\mathcal{F}^+(E;E)}$. Moreover, as ${\rm id}_E - R_N \geq 0$, we have
    $$ \lim_{N \to \infty }\norma{|R_N - {\rm id}_E |(x)} = \lim_{N \to \infty } \norma{x - R_N(x)} = 0, $$
    proving that $E$ has the SPAP.

    The implication (2)$\Rightarrow$(4) follows from \cite[Remark 4.3]{bupams}. It remains us to check that (4)$\Rightarrow$(2). Indeed, assuming that $E$ is a separable Banach lattice with the RAP, we only need to check that $E$ is atomic and has order continuous norm, which is equivalent to prove that all order intervals of the form $[-u,u]$, with $u \in E^+$, are totally bounded in $E$. Given $\varepsilon > 0$ and applying the RAP to $C = \{u\}$, there exists $T \in \mathcal{F}(E; E)$ such that
    $ \norma{|T - {\rm id}_E|(u)} < \dfrac{\varepsilon}{2}. $ Since $T(E)$ is finite dimensional, $T([-u,u])$ is totally bounded. Following the same argument as in the proof of the implication (1)$\Rightarrow$(2) we obtain that $[-u,u]$ is totally bounded, as desired.
\end{proof}

The main objective in this section is to provide a generalization of Corollary \ref{corret} in the setting of multilinear operators:

\begin{theorem} \label{teoret1}
    Let $E_1, \dots, E_n$ and $F$ be Banach lattices with $F$ being Dedekind complete such that every positive $n$-linear operator from $E_1 \times \cdots \times E_n$ into $F$ attains its norm. Assume that for every $A \in B_{{\cal L}^+(E_1, \dots, E_n; F)}$, there exists a sequence operators $(A_k)_k \subset B_{{\cal F}^+(E_1, \dots, E_n; F)}$ such that 
    $$ \lim_{k \to \infty} \norma{|A_k - A|(x_1, \dots, x_n)} = 0$$
    for all $x_1 \in E_1^+, \dots, x_n \in E_n^+$. Then,     every positive $n$-linear operator from $E_1 \times \cdots \times E_n$ into $F$ is weakly sequentially continuous.
\end{theorem}

To prove Theorem \ref{teoret1}, we need to recall that the positive projective tensor product, or the Fremlin tensor product, of $E_1, \dots, E_n$, denoted by $E_1 \widehat{\otimes}_{|\pi|} \cdots \widehat{\otimes}_{|\pi|} E_n$, is the Banach lattice generated by considering the completion of the vector lattice tensor product $E_1 \overline{\otimes} \cdots \overline{\otimes} \, E_n$ with respect the positive projective tensor norm:
$$ \norma{u}_{|\pi|} = \inf \conj{\sum_{i=1}^k \norma{x_{1,i}}\cdots \norma{ x_{n,i} }}{|u| \leq \sum_{i=1}^k x_{1,i} \otimes \cdots \otimes x_{n,i}, \, x_{j,i} \in E_j^+}.$$
For each $A \in \mathcal{L}^r(E_1, \dots, E_n; F)$, there exists a unique bounded linear operator $A^\otimes$ from  $E_1 \widehat{\otimes}_{|\pi|} \cdots \widehat{\otimes}_{|\pi|} E_n$ into $F$, called the linearization of $A$, such that $A^\otimes (x_1 \otimes \cdots \otimes x_n) = A(x_1, \dots, x_n)$ for all $x_1 \in E_1, \dots, x_n \in E_n$. Moreover, the map
$$ A \in \mathcal{L}^r(E_1, \dots, E_n; F) \mapsto A^\otimes \in \mathcal{L}^r(E_1 \widehat{\otimes}_{|\pi|} \cdots \widehat{\otimes}_{|\pi|} E_n; F) $$
is an isometric isomorphism. If, in addition, $F$ is Dedekind complete,  the correspondence is also a lattice homomorphism (see \cite[Proposition 3.3]{bubuskes}).

\medskip

\noindent{\it Proof of Theorem \ref{teoret1}}
Let $A: E_1 \times \cdots \times E_n \to F$ be a positive $n$-linear operator with $\norma{A} \leq 1$. By the assumption, there exists a sequence of operators $(A_k)_k \subset B_{{\cal F}^+(E_1, \dots, E_n; F)}$ such that 
$$ \lim_{k \to \infty} \norma{|A_k - A|(x_1, \dots, x_n)} = 0 $$
for all $x_1 \in E_1^+, \dots, x_n \in E_n^+$.  We claim that $(A_k^\otimes)_k$ converges to $A^\otimes$ in the $|SOT|$ of ${\cal L}(E_1 \widehat{\otimes}_{|\pi|} \cdots \widehat{\otimes}_{|\pi|} E_n; F)$. To see this, let $z\in (E_{1}\widehat{\otimes}_{|\pi|}\cdots \widehat{\otimes}_{|\pi|} E_{n})^{+}$ and $\varepsilon>0$ be given.  Since the cone generated by $\conj{x_1 \otimes \cdots \otimes x_n}{x_1 \in E_1^+, \dots, x_n \in E_n^+}$ is dense in $E_{1}\widehat{\otimes}_{|\pi|}\cdots \widehat{\otimes}_{|\pi|} E_{n}$ (see \cite[p. 850]{bubuskes}), there exists $w=\displaystyle\sum_{i=1}^{m} x_{1,i}\otimes\cdots\otimes x_{n,i}$ with $x_{j,i} \geq 0$ for all $i$ and $j$ such that $\|z-w\|_{|\pi|}< \dfrac{\varepsilon}{4}.$
Now, using the pointwise convergence of $|A_k - A|$, we can find $k_0 \in \N$ such that
$$\||A_{k}-A|(x_{1,i},\ldots,x_{n,i})\|<\dfrac{\varepsilon}{2m}$$ 
for every $k>k_{0}$ and every $i =1, \dots, m$. Thus, for every $k>n_{0}$,
\begin{align*}
    \norma{|A_k - A|^\otimes(w)} & \leq \sum_{i=1}^m \norma{|A_k - A|^\otimes (x_{1,i}\otimes\cdots\otimes x_{n,i})} < \frac{\varepsilon}{2}.
\end{align*}
Since the correspondence $A\mapsto A^{\otimes}$ is an isometric isomorphism and a lattice homomorphism, we have 
\begin{align*}
    \||A_{k}^{\otimes}-A^{\otimes}|(z)\|&=\||A_{k}-A|^{\otimes}(z)\| \leq \||A_{k}-A|^{\otimes}(|z-w|)\|+\||A_{k}-A|^{\otimes}(w)\|\\
    &\leq \||A_{k}-A|^{\otimes}\|\|z-w\|_{|\pi|} +\dfrac{\varepsilon}{2} < \varepsilon,
\end{align*}
proving that $A_k^\otimes \to A^\otimes$ in the $|SOT|$.

Let $C := {\cal L}_{wsc}^+(E_1, \dots, E_n; F)$ and define $C^\otimes := \conj{S^\otimes}{S \in C}$. We notice that $C^\otimes$ is a norm-closed convex subset of ${\cal L}(E_1 \widehat{\otimes}_{|\pi|} \cdots \widehat{\otimes}_{|\pi|} E_n; F)$. Besides, since we are assuming that every positive $n$-linear operator from $E_1 \times \cdots \times E_n$ into $F$ attains its norm, we get that every positive linear operator from $E_1 \widehat{\otimes}_{|\pi|} \cdots \widehat{\otimes}_{|\pi|} E_n$ into $F$ is norm-attaining. Thus, it follows from \cite[Lemma 2.9]{luizmiranda} that $C^\otimes$ is sequentially $|SOT|$-closed, which proves that $A^\otimes \in C^\otimes$. Therefore, $A \in C$ and we are done. \qed

\medskip

As a consequence, we get the following:

\begin{theorem} \label{teoret2}
    Let $E_1, \dots, E_n$ be reflexive Banach lattices and let $F$ be a Dedekind complete Banach lattice. Assume that for every $A \in B_{{\cal L}^+(E_1, \dots, E_n; F)}$, there exists a sequence operators $(A_k)_k \subset B_{{\cal F}^+(E_{1},\ldots,E_{n}; F)}$ such that 
    $$ \lim_{n \to \infty} \norma{|A_k - A|(x_1, \dots, x_n)} = 0$$
    for all $x_1 \in E_1^{+}, \dots, x_n \in E_n^{+}$.
    Then, the following are equivalent:\\
    {\rm (1)} Every positive $n$-linear operator  $A: E_1 \times \cdots \times E_n \to F$ is weakly sequentially continuous. \\
    {\rm (2)} Every positive $n$-linear operator $A: E_1 \times \cdots \times E_n \to F$ attains its norm.
\end{theorem}

In order to provide illustrative examples of Theorems \ref{teoret1} and \ref{teoret2}, we will need the following proposition.

\begin{proposition}  \label{propapret}
    Let $E_1, \dots, E_n$ and $F$ be Banach lattices with $F$  being Dedekind complete. Suppose that one of the following conditions holds: \\
    {\rm (1)} $E_1, \dots, E_n$ have the SPAP.\\
    {\rm (2)} $F$ and at least $n-1$ among $E_1, \dots, E_n$ have the SPAP.\\
    Then, for every $A \in B_{{\cal L}^+(E_1, \dots, E_n; F)}$ there exists a sequence $(A_{k})_{k} \subset B_{\mathcal{F}^+(E_1, \dots, E_n; F)}$ such that
    $$ \lim_{k \to \infty} \norma{|A_{k}-A|(x_1, \dots, x_n)} = 0 $$
    for all $x_1 \in E_1^+,\dots, x_n \in E_n^+$.
\end{proposition}
\begin{proof} Let $0 \neq A \in B_{\mathcal{L}^+(E_1, \dots,  E_n; F)}$. 

(1) We assume first that $E_1, \dots, E_n$ have the SPAP. For the convenience of the reader, we prove for $n=2$.
Thus, the assumption ensures the existence of two sequences $(T_{1,k})_k \subset B_{{\cal F}^+(E_1; E_1)}$
and $(T_{2,k})_k \subset B_{{\cal F}^+(E_2; E_2)}$ such that 
$$ \lim_{k\to \infty} \norma{|T_{i,k} - {\rm id}_{E_i}|(x_i)} = 0 $$
for every $x_i \in E_i^+$ and every $i \in \{ 1,2\}$.
For each $k \in \N$, defining 
$$ A_k(x_1, x_2) := A(T_{1,k}(x_1), T_{2,k}(x_2)), \quad x_1 \in E_1, \, x_2 \in E_2, $$
we get that $(A_k)_k$ induces a sequence of positive bilinear operators from $E_1\times E_2$ into $F$. Also, the same argument used in the proof of Proposition \ref{propap} shows that $(A_k)_k \subset B_{{\cal F}^+(E_1, E_2; F)}$.
Now, we check that $|A_k - A|(x_1, x_2) \longrightarrow 0$ for all $x_1 \in E_1^+$ and $x_2 \in E_2^+$. First, notice that, for each $z \in E_1^+$ and each $w \in E_2^+$,
\begin{align*}
    |(A_k-A)(z,w)|&=|A(T_{1,k}(z),T_{2,k}(w))-A(z,w)|\\
   &\leq 
   |A(T_{1,k}(z), T_{2,k}(w) - w)| +|A(T_{1,k}(z) - z, w)|\\
    &\leq A(T_{1,k}(z),|T_{2,k}(w)-w|)+A(|T_{1,k}(z)-z|,w) \\
    &\leq A(T_{1,k}(z),|T_{2,k}-{\rm id}_{E_2}|(w))+A(|T_{1,k}-{\rm id}_{E_1}|(z),w). 
\end{align*}
Now, let $x_1 \in E_1^+$ and $x_2 \in E_2^+$ be fixed. Given $z_1, \dots, z_l \in E_1^+$ and $w_1, \dots, w_m \in E_2^+$ such that $x_1 = \displaystyle \sum_{i=1}^l z_i$ and $x_2 = \displaystyle \sum_{j=1}^m w_j$, using the above inequality, we get that
\begin{align*}
    \sum_{i=1}^l \sum_{j=1}^m \left| (A_k - A)(z_i, w_j) \right| & \leq 
\sum_{i=1}^l \sum_{j=1}^m \Big[ 
A\big(T_{1,k}(z_i), |T_{2,k} - \mathrm{id}_{E_2}|(w_j)\big) 
+ 
A\big(|T_{1,k} - \mathrm{id}_{E_1}|(z_i), w_j\big)
\Big] \\
& = A\big(T_{1,k}(x_1), |T_{2,k} - \mathrm{id}_{E_2}|(x_2)\big) 
+ 
A\big(|T_{1,k} - \mathrm{id}_{E_1}|(x_1), x_2\big).
\end{align*}
Taking the supremum of all such $z_1, \dots, z_l$ and $w_1, \dots, w_m$, we obtain from \cite[Proposition 2.14]{loane} that
$$ |A_k - A|(x_1, x_2) 
\leq 
 A\big(T_{1,k}(x_1), |T_{2,k} - \mathrm{id}_{E_2}|(x_2)\big) 
+ 
A\big(|T_{1,k} - \mathrm{id}_{E_1}|(x_1), x_2\big). $$
Thus,
    \begin{align*}
        \norma{|A_k-A|(x_{1},x_{2})} & \leq \norma{ A\big(T_{1,k}(x_1), |T_{2,k} - \mathrm{id}_{E_2}|(x_2)\big) }
+ 
\norma{A\big(|T_{1,k} - \mathrm{id}_{E_1}|(x_1), x_2\big)} \\
        & \leq \norma{x_1} \norma{|T_{2,k}-{\rm id}_{E_2}|(x_2)} + \norma{|T_{1,k}-{\rm id}_{E_1}|(x_1)} \norma{x_2} \longrightarrow 0.
    \end{align*}

 We also prove item (2) for $n=2$. Assume that $F$ and $E_1$ have the SPAP. Then, there exist sequences $(T_k)_k \subset B_{\mathcal{F}^+(E_1; E_1)}$ and $(R_k)_k \subset B_{\mathcal{F}^+(F; F)}$ such that
$$ \lim_{k \to \infty} \norma{|T_k - {\rm id}_{E_1}|(x_1)} = 0 \quad \text{ and } \quad  \lim_{k \to \infty} \norma{|R_k - {\rm id}_{F}|(y)}= 0$$ for every $x_1 \in E_1^+$ and every $y \in F ^+$. For each $k \in \N$, define $A_k(x_1, x_2) := R_k(A(T_k(x_1), x_2))$ for all $x_1 \in E_1$ and $x_2 \in E_2$. Clearly, each $A_k$ is positive with $\norma{A_k} \leq 1$. Moreover, the same proof used in item (2) of Proposition \ref{propap} implies that $(A_k)_k \subset B_{\mathcal{F}^+(E_1, E_2; F)}$. We now prove that
$$ \lim_{k \to \infty} \norma{|A_k - A|(x_1, x_2)} = 0 $$
for all $x_1 \in E_1^+$ and $x_2 \in E_2^+$. To see this, fix $x_1 \in E_1^+$ and $x_2 \in E_2^+$. Given $z_1, \dots, z_l \in E_1^+$ and $w_1, \dots, w_m \in E_2^+$ such that $x_1 = \displaystyle \sum_{i=1}^l z_i$ and $x_2 = \displaystyle \sum_{j=1}^m w_j$, we have
\begin{align*}
    \sum_{i=1}^l \sum_{j=1}^m |(A_k - &A)(z_i, w_j)|  \leq  \sum_{i=1}^l \sum_{j=1}^m [ |(R_k - {\rm id}_F)(A(T_k(z_i), w_j))| + |A(T_k(z_i) - z_i, w_j)| ] \\
    & \leq |R_k - {\rm id}_F| \left ( \sum_{i=1}^l \sum_{j=1}^m A(T_k(z_i), w_j))  \right ) +  \sum_{i=1}^l \sum_{j=1}^m  A(|T_k - {\rm id}_{E_1}|(z_i), w_{j}) \\
    & = |R_k - {\rm id}_F|(A(T_k(x_1), x_2)) + A(|T_k - {\rm id}_{E_1}|(x_1), x_2).
\end{align*}
Taking the supremum over all such decompositions and using \cite[Proposition 2.14]{loane}, we obtain 
$$ |A_k - A|(x_1, x_2) \leq |R_k - {\rm id}_F|(A(T_k(x_1), x_2)) + A(|T_k - {\rm id}_{E_1}|(x_1), x_2), $$
and therefore
\begin{align*}
    \norma{|A_k - A|(x_1, x_2)} & \leq \norma{|R_k - {\rm id}_F|(A(T_k(x_1), x_2))} + \norma{A(|T_k - {\rm id}_{E_1}|(x_1), x_2)} \\
    & \leq \norma{|R_k - {\rm id}_F|(A(T_k(x_1), x_2) - A(x_1, x_2))} + \norma{|R_k - {\rm id}_F|(A(x_1, x_2))} + \\
    & + \norma{A(|T_k - {\rm id}_{E_1}|(x_1), x_2)} \\
    & \leq 3 \norma{A} \norma{|T_k - {\rm id}_{E_1}|(x_1)} \norma{x_2} +   \norma{|R_k - {\rm id}_F|(A(x_1, x_2))} \to 0,
\end{align*}
and we are done.
\end{proof}

\begin{examples}  \label{exret} \rm
    {\rm (1)}  Let $1 < p_1, \dots, p_n < \infty$ be such that $\displaystyle \sum_{i=1}^n \frac{1}{p_i} < 1$. By \cite[Proposition 4.1]{alencarfloret}, all $n$-linear continuous operators from $\ell_{p_1} \times \cdots \times \ell_{p_n}$ into $\R$ are weakly sequentially continuous. Thus, by Lemma \ref{lemmafin}, for any Banach space $F$, every compact $n$-linear operator from $\ell_{p_1} \times \cdots \times \ell_{p_n}$ into $F$ is weakly sequentially continuous. We can use \cite[Theorem 4.4]{botmir} to obtain Banach lattices $F$ such that all positive $n$-linear operators from $\ell_{p_1} \times \cdots \times \ell_{p_n}$ into $F$ are compact, hence weakly sequentially continuous, hence norm-attaining by Theorem \ref{teoret2}. For instance, we can consider: 
    
    (i) $F = L_q(\nu)$ with $1 \leq q <\left ( \displaystyle \sum_{i=1}^n \frac{1}{p_i} \right )^{-1}$; 
    
    (ii) $F = ({\rm FBL}[L_{q}(\nu)])^* $
with $2 \leq q < \infty$ and $\displaystyle \sum\limits_{i=1}^n \frac{1}{p_i} < \frac{1}{2}$, where ${\rm FBL}[L_{q}(\nu)]$ denotes the Free Banach lattice of $L_q(\nu)$ (see, e.g, \cite{aviles}).
    
  \noindent (2) It follows from \cite[Example 1.1]{botmir} along with Theorem \ref{maintheo} that there are non-norm attaining bilinear operators from $\ell_4 \times \ell_4$ into $L_1([0,1])$. However, as we just saw above, all positive bilinear operators from $\ell_4 \times \ell_4$ into $L_1([0,1])$ are norm-attaining.\\
\noindent (3) Let $E_{1},\ldots,E_{n}$ be AM-spaces with units $e_{1},\ldots,e_{n}$, respectively, and let $F$ be a Dedekind complete Banach lattice. We claim that all $n$-linear positive operators $A\colon E_{1}\times\cdots\times E_{n}\longrightarrow F$ are norm attaining. We illustrate the argument in the case $n=2$. Indeed, consider the isometric isomorphism $\psi:\mathcal{L}^r (E_{1},E_{2};F)\longrightarrow\mathcal{L}^{r}(E_{1};\mathcal{L}^r(E_{2};F))$, given by 
$$ \psi(A)(x_1)(x_2) = A(x_1, x_2), \, A \in \mathcal{L}^r(E_{1},E_{2};F), \, x_1 \in E_1, \, x_2 \in E_2. $$
So, for every $A \in \mathcal{L}^+ (E_{1},E_{2};F) $, $\psi(A)$ is a positive linear operator defined on $E_1$. As $E_1$ is an AM-space with unit $e_1$, we have from \cite[Exercise 2, p. 270]{positiveoperators} that
$$ \norma{\psi(A)(e_1)} = \norma{\psi(A)} = \norma{A}.  $$
Furthermore, as $\psi(A)(e_1)$ is a positive linear operator on $E_2$ that is also an $AM$-space with unit $e_2$, we get again from \cite[Exercise 2, p. 270]{positiveoperators} that
$$ \norma{A(e_1,e_2)} = \norma{\psi(A)(e_1)(e_2)} = \norma{\psi(A)(e_1)} = \norma{A}, $$
proving that all bilinear positive operators $A:E_1 \times E_2 \to F$ are norm-attaining. The general case will follow inductively using the isometric identification between $\mathcal{L}^r(E_1, \dots, E_n; F)$ and $\mathcal{L}^r(E_1, \mathcal{L}^{r}(E_2, \dots, E_n; F))$. \\
\noindent (4) Since infinite dimensional AM-spaces are not reflexive, we can always find non-norm attaining $n$-linear operators defined in the product of AM-spaces. 
\\
\noindent (5)  Now, we shows that, in Proposition \ref{propapret}(2), the assumption that at least $n-1$ of the domain spaces have the SPAP cannot be omitted, even when the range space has the SPAP. 
Let $E=L_2([0,1])$ and consider the continuous positive bilinear form
$$ A\colon E\times E\longrightarrow\mathbb R, \quad
A(f,g)=\int_0^1 f(t)g(t)\,dt. $$
Clearly, $\norma{A} = 1$. 
We claim that
$|A-B|(\mathbf 1,\mathbf 1)\geq 1$  for every $B\in\mathcal F^+(E,E;\mathbb R)$, where $\mathbf 1$ denotes the constant function equal to one. Indeed, fix $B\in\mathcal F^+(E,E;\mathbb R)$. Since $B$ has finite type, there exist
$\varphi_1,\ldots,\varphi_m,\psi_1,\ldots,\psi_m\in E^*$ such that $$ B(f,g)=\sum_{j=1}^m\varphi_j(f)\psi_j(g), \qquad f,g\in E. $$
If $(r_k)_k$ is the Rademacher sequence, then $r_k \cvf 0$ in $E$, and so
$$ |A(r_k,r_k) - B(r_k, r_k)| = \Big|1 -  \displaystyle \sum_{j=1}^m\varphi_j(r_k)\psi_j(r_k)\Big| \longrightarrow  1.$$
By taking $k \to \infty$, in the inequality below
$$  |A-B|(\boldsymbol{1}, \boldsymbol{1}) = |A-B|(|r_k|, |r_k|)  \geq |(A-B)(r_k, r_k)|, $$
we have that $|A-B|(\boldsymbol{1}, \boldsymbol{1}) \geq 1$. Therefore, no sequence $(A_k)_k \subset B_{\mathcal{F}^+(E, E; \R)}$ can satisfy $\displaystyle \lim_{k \to \infty} |A_k - A|(f,g) = 0$ for all $f,g \in E^{+}$.
\end{examples}

We now proceed to the polynomial case. First, we recall that the $n$-fold positive projective symmetric tensor product of $E$, denoted by
$\widehat{\otimes}_{n, s, |\pi|} E$,
is a Banach lattice endowed with the positive projective symmetric tensor
norm: 
$$
\norma{u}_{s, |\pi|} = \inf \conj{\sum_{i=1}^k \norma{x_{i}}^n}{|u| \leq
  \sum_{i=1}^k \otimes^n x_i, \, x_i \in E^+},
$$
where  $\otimes^n x = x \, \otimes \stackrel{n}{\cdots} \otimes \, x$ for every $x \in E$.  Let  $\theta_n\colon E \to \widehat{\otimes}_{n, s, |\pi|} E$ be the canonical $n$-homogeneous polynomial given by $\theta_n(x) =\otimes^n x$. We note that $\theta_n$
is a lattice homomorphism. For every $P \in \mathcal{P}^r(^n E, F)$ there
exists a unique regular linear operator $P^\otimes\colon \widehat{\otimes}_
{n, s,|\pi|} E \to  F$, called the linearization of $P$, such that $P(x) =
P^\otimes(\theta_n( x))$ for every $x \in E$. 
 The operator
$$\Phi\colon  \mathcal{P}^r(^n E; F) \to  \mathcal{L}^r\left(\widehat{\otimes}_{n, s, |\pi|} E; F\right),~ \Phi(P) = P^\otimes, $$ is an isometric isomorphism and a lattice homomorphism \cite[Proposition 3.4]{bubuskes}. Making the necessary adaptations, we obtain the following polynomial versions of Theorem \ref{teoret1}, Theorem \ref{teoret2}, and Proposition \ref{propapret}, respectively. The proofs are omitted, as they are analogous to those of the previous results.

\begin{theorem} \label{teoretpol1}
    Let $E$ and $F$ be Banach lattices  such that  $F$ is Dedekind complete and  every $n$-homogeneous positive polynomial from $E$ to $F$ attains its norm. Assume that for every $P\in B_{\mathcal{P}^{+}(^nE;F)}$ there exists a sequence $(P_{k})_k\subset B_{\mathcal{P}^+(^nE;F)}$ such that $T_{P_k}$ has finite type for each $k \in \N$ and 
    $$\lim_{k\to \infty} \||P_{k}-P|(x)\|=0$$
    for every $x\in E^{+}$. Then, every $n$-homogeneous positive polynomial from $E$ to $F$ is weakly sequentially continuous.
\end{theorem}

 \begin{theorem} \label{teoretpol2}
    Let $E$ be a reflexive Banach lattice and $F$ be a Dedekind complete Banach lattice. Assume that for every $P\in B_{\mathcal{P}^{+}(^nE;F)}$ there exists a sequence $(P_{k})_k\subset  B_{\mathcal{P}^+(^nE;F)}$ such that $T_{P_k}$ has finite type for each $k \in \N$ and
    $$\lim_{k\to \infty} \||P_{k}-P|(x)\|=0$$
    for every $x\in E^{+}$. Then, the following statements are equivalent: \\
    {\rm (1)} Every positive $n$-homogeneous polynomial from $E$ to $F$ attains its norm. \\
       {\rm (2)} Every positive $n$-homogeneous polynomial from $E$ to $F$ is weakly sequentially continuous.
\end{theorem}

\begin{proposition} \label{apretpol}
    Let $E$ and $F$ be Banach lattices with $F$ being Dedekind complete. If $E$ has the SPAP, then, for every
    $P\in B_{\mathcal{P}^{+}(^nE;F)}$ there exists a sequence $(P_{k})_k\subset B_{\mathcal{P}^+(^nE;F)}$ such that $T_{P_k}$ has finite type for each $k \in \N$ and
    $$\lim_{k\to \infty} \||P_{k}-P|(x)\|=0$$
    for every $x\in E^{+}$.
\end{proposition}

\begin{proof}
    Suppose that $E$ has the SPAP.  For the sake of clarity, we prove the proposition for $n=2$. Let $P\in B_{\mathcal{P}^{+}(^{2}E;F)}$. By assumption,  there exists a sequence $(S_{k})_k\subset B_{\mathcal{F}^{+}(E;E)}$ such that
    $$\lim_{k\to\infty}\||S_{k}-{\rm id}_E|(x)\|=0,$$
    for every $x\in E^{+}$. For each $k\in\mathbb{N}$, define $A_{k}\colon E\times E\longrightarrow F$  by $$A_{k}(x_{1},x_{2})=T_{P}(S_{k}(x_{1}),S_{k}(x_{2})).$$ 
    Clearly, each $A_k$ is a positive bilinear operator; it is symmetric since $T_P$ is symmetric, and of finite type since $S_k$ is of finite type.
     Let $P_{k}$ be the positive $2$-homogeneous polynomial such that $P_{k}(x)=T_{P}((S_{k}(x)),(S_{k}(x)))=:T_{P_{k}}(x,x)$, for all $x\in E$.  Note that, for each $k\in\mathbb{N}$, $Q_{k}=P_{k}-P$ is a regular polynomial and its associated symmetric bilinear operator is $A_{k}-T_P$. Given $x, z_1,\ldots,z_r, w_1, \dots w_l \in E^+$ such that $\displaystyle \sum_{i=1}^r z_i=x = \displaystyle \sum_{j=1}^l w_j$, the same argument used in the proof of Proposition \ref{propapret} yields that  
     \begin{align*}
         \sum_{i, j=1}^{r,l} |(A_k - T_P)(z_i, w_j)| &  \leq \sum_{i,j = 1}^{r,l}  \Big[T_P(S_k(z_i), |S_k - {\rm id}_E|(w_j)) + T_P(|S_k - {\rm id_{E}}|(z_i), w_j)\Big]\\
         & = T_P(S_k(x), |S_k - {\rm id}_E|(x)) + T_P(|S_k - {\rm id}_E|(x), x) \\
         & = T_P(|S_k - {\rm id}_E|(x), S_k(x) + x) 
     \end{align*}
     Taking the supremum of all such $z_{1},\ldots,z_r,w_1, \dots, w_l$, we get from \cite[Lemma 2.16]{loane} that 
     $$|P_{k}-P|(x)\leq T_P(|S_k - {\rm id}_E|(x), S_k(x) + x)$$
     for every $x\in E^{+}$. Hence, 
     \begin{align*}
         \||P_{k}-P|(x)\| & \leq \norma{T_P(|S_k - {\rm id}_E|(x), S_k(x) + x}) \\
         & \leq \|T_P\|\||S_{k}-id_E|(x)\| \|S_k(x) + x \| \\
         & \leq \norma{T_P} \||S_{k}-id_E|(x)\| 2 \norma{x} \longrightarrow 0.
     \end{align*}

\end{proof}

\begin{example} \rm
   (1)  Let $n \in \N$ and $p > n$ be given. By \cite[4.3]{alencarfloret}, all $n$-homogeneous polynomials from $\ell_p$ into $\R$ are weakly sequentially continuous. Thus, by \cite[Lemma 2.1]{bu2015}, for any Banach space $F$, every compact $n$-homogeneous polynomial from $\ell_{p}$ into $F$ is weakly sequentially continuous. We can use \cite[Theorem 4.4]{botmir} to obtain Banach lattices $F$ such that all positive $n$-homogeneous polynomials from $\ell_{p}$ into $F$ are compact, hence weakly sequentially continuous, hence norm-attaining by Theorem \ref{teoretpol2}.   \\
   (2) Recall that $P((a_j)_j) = \displaystyle \sum_{j=1}^\infty a_j^2 r_j$, where $(r_j)_j$ denotes the sequence of the Rademacher functions, defines a non-compact $2$-homogeneous polynomial. Thus, as $\ell_4$ is reflexive, $P$ cannot be weakly sequentially continuous, and so by Theorem \ref{maintheopol}, there are non-norm attaining $2$-homogeneous polynomials from $\ell_4$ into $L_1([0,1])$. However, by item (1) above, all positive $2$-homogeneous polynomials from $\ell_4$ into $L_1([0,1])$ are norm-attaining.  \\
   (3) Let $E$ be an AM-space with a unit $e$ and $F$ be a Dedekind complete Banach lattice. We claim that every $P \in \mathcal{P}^{+}(^n E; F)$ attains its norm in $e$. Indeed, given $x_1, \dots, x_n \in E$ with $x_i \neq 0$ for every $i$, we get that
   \begin{align*}
       \norma{ T_P(x_1, \dots, x_n)} & = \norma{T_P(\frac{x_1}{\norma{x_1}}, \dots, \frac{x_n}{\norma{x_n}})} \norma{x_1} \cdots \norma{x_n} \\
       & \leq \norma{T_P(e, \dots, e)} \norma{x_1} \cdots \norma{x_n} \\
       & = \norma{P(e)} \norma{x_1} \cdots \norma{x_n} \\
       & \leq \norma{P} \norma{x_1} \cdots \norma{x_n},
   \end{align*}
which implies that $\norma{T_P} \leq \norma{P}$, and so $\norma{T_P}=\|P\|$. Recalling that $T_P$ attains its norm in $(e, \dots, e)$ by Examples \ref{exret}, we obtain
$ \norma{P(e)} = \norma{T_P(e, \dots, e)} = \norma{T_P} = \norma{P}. $
\end{example}

We notice that while Theorems \ref{maintheo} and \ref{maintheopol} were stated with an additional condition, their lattice versions presented in this Section were stated without this third equivalence. In order to obtain it, we need to assume that $F$ is an atomic Banach lattice with order continuous norm. Below, we prove these new results for  a more stricter class:

\begin{corollary} \label{corfin1}
    Let $E_1, \dots, E_n$ be a reflexive Banach lattice and let $F$ be a Banach lattice with order-continuous norm and order given by a basis. Assume, in addition, that at least $n-1$ among $E_1, \dots, E_n$ have the SPAP.
    Then, the following are equivalent: \\
    {\rm (1)} Every positive $n$-linear operator $T: E_1 \times \cdots \times E_n \to F$ is weakly sequentially continuous. \\
    {\rm (2)} Every positive $n$-linear operator $T: E_1 \times \cdots \times E_n \to F$ attains its norm.\\
    {\rm (3)} All positive $n$-linear forms $E_1 \times \cdots \times E_n \to \R$ are weakly sequentially continuous and all positive $n$-linear operators operators $E_1 \times \cdots \times  E_n \to F$ are compact.
\end{corollary}

\begin{proof}
    Notice that our assumption in $F$ implies that $F$ has the SPAP, and so we get (1)$\iff$(2). The implication (1)$\Rightarrow$(3) follows by the same argument used in the proof of Theorem \ref{maintheo}.
    
    (3)$\Rightarrow$(1)     Let $A: E_1 \times \cdots \times E_n \to F$ be a positive $n$-linear operator and let, for each $j =1, \dots, n$, $(x_{j,k})_k$ be a weakly convergent sequence in $E_j$ with weak limit $x_j$. We prove first that the mapping $\otimes_n: E_1 \times \cdots \times E_n \to E_1 \widehat{\otimes}_{|\pi|} \cdots \widehat{\otimes}_{|\pi|} E_n$ satisfies
    $$x_{1,k} \otimes \cdots \otimes x_{n,k} \cvf x_1 \otimes \cdots \otimes x_n$$
    in $E_1 \widehat{\otimes}_{|\pi|} \cdots \widehat{\otimes}_{|\pi|} E_n$. To see this, consider a positive linear functional $\varphi \in (E_1 \widehat{\otimes}_{|\pi|} \cdots \widehat{\otimes}_{|\pi|} E_n)^*$ and the positive $n$-linear form $T: E_1 \times \cdots \times E_n \to \R$ such that $T^\otimes = \varphi$. Our assumption yields that $T$ is weakly sequentially continuous, and so 
    $$ \lim_{k \to \infty} \varphi(x_{1,k} \otimes \cdots \otimes x_{n,k}) = \lim_{k \to \infty} T(x_{1,k}, \dots, x_{n,k}) = T(x_1, \dots, x_n) = \varphi(x_1 \otimes \dots \otimes x_n). $$
    From this it is easy to check that $x_{1,k} \otimes \cdots \otimes x_{n,k} \cvf x_1 \otimes \cdots \otimes x_n$
    in $E_1 \widehat{\otimes}_{|\pi|} \cdots \widehat{\otimes}_{|\pi|} E_n$. On the other hand, it follows from \cite[Lemma 3.3]{bupams}
    that 
    $$ A^\otimes (B_{E_1 \widehat{\otimes}_{|\pi|} \cdots \widehat{\otimes}_{|\pi|} E_n}) \subset \overline{\rm sco}(A(B_{E_1}^+ \times \cdots \times B_{E_n}^+)), $$
    where $\overline{\rm sco} (S)$ denotes the closed solid-convex hull of a subset $S$ of a Banach lattice. 
    Notice that our second assumption yields that $A$ is compact, and so $A(B_{E_1}^+ \times \cdots \times B_{E_n}^+)$ is a relatively compact subset of $F$. As we are assuming that the order of $F$ is continuous and given by a basis, it is an atomic Banach lattice with order continuous norm, and so the solid hull of relatively compact subsets of $F$ are also relatively compact (see \cite[Theorem 5]{wickstead}). Thus, $A^\otimes: E_1 \widehat{\otimes}_{|\pi|} \cdots \widehat{\otimes}_{|\pi|} E_n \to F$ is a compact operator, hence completely continuous. Therefore, 
    $$ \lim_{k \to \infty}A(x_{1,k}, \dots, x_{n,k}) = \lim_{k \to \infty} A^{\otimes}(x_{1,k} \otimes \cdots \otimes x_{n,k}) = A^\otimes(x_1 \otimes \cdots \otimes x_n) = A(x_1, \dots, x_n), $$
proving that $A$ is weakly sequentially continuous.
\end{proof}

The polynomial version of Corollary \ref{corfin1} follows analogously using the compactness of the linearization of a positive homogeneous polynomial given by Li and Bu \cite[Theorem 4.1]{libumaj}:

\begin{corollary} \label{corfin2}
    Let $E$ be a reflexive Banach lattice with the SPAP and let $F$ be a Banach lattice with order-continuous norm and order given by a basis. Then, the following are equivalent: \\
    {\rm (1)} Every positive $n$-homogeneous polynomial $P: E \to F$ is weakly sequentially continuous. \\
    {\rm (2)} Every positive $n$-homogeneous polynomial $P:E \to F$ is norm attaining. \\
    {\rm (3)} All positive $n$-linear polynomials $E \to \R$ are weakly sequentially continuous and all positive $n$-homogeneous polynomials $E \to F$ are compact. 
\end{corollary}

\section{Further remarks and open questions}

In this final section, we present and briefly discuss a few open problems motivated by the results of this paper.

In view of Theorems \ref{maintheo} and \ref{maintheopol}, together with the examples presented in Section 2, we are led to the following question.

\medskip

\noindent {\bf Question 1:} Are there Banach spaces $X$ and $Y$ such that every $n$-homogeneous polynomial $P: X \to Y$ is norm-attaining, but there exists a non norm-attaining $n$-linear operator $A: X^n \to Y$?

\medskip

In connection with the previous question, we have the following easy consequence of Theorems \ref{maintheo} and \ref{maintheopol} which highlights the relationships between the multilinear and polynomial settings.

\begin{corollary} \label{corprob}
Let $n \in \N$, $X$ be a reflexive Banach space and  $Y$ be a Banach spaces. Consider the following implications: \\
{\rm (1)} Every $A \in \mathcal{L}(X^n; Y)$ is norm-attaining. \\
{\rm (2)} Every $A \in \mathcal{L}(X^n; Y)$ is weakly sequentially continuous. \\
{\rm (3)} Every $P \in \mathcal{P}(^n X; Y)$ is weakly sequentially continuous. \\
{\rm (4)} Every $P \in \mathcal{P}(^n X; Y)$ is norm-attaining. \\
If $X$ has the BAP, then $(1)\iff (2) \Rightarrow (3) \iff (4)$.
\end{corollary}

We point out that the implication (3)$\Rightarrow$(2) of Corollary \ref{corprob} coincides with Problem 1 left by Castillo, García, and Gonzalo in \cite{castillo99} that, as far as we know, remains open. One could also address this same question to the Banach lattice case. In particular, we state the lattice version of Corollary \ref{corprob}:

\begin{corollary} \label{corprob2}
    Let $n \in \N$, $E$ be a reflexive Banach lattice and $F$ be a Dedekind complete Banach lattice. Consider the following implications: \\
    {\rm (1)} Every $A \in \mathcal{L}^+(E^n; F)$ is norm-attaining. \\
     {\rm (2)} Every $A \in \mathcal{L}^+(E^n; F)$ is weakly sequentially continuous. \\
     {\rm (3)} Every $P \in \mathcal{P}^+(E^n; F)$ is weakly sequentially continuous. \\
     {\rm (4)} Every $P \in \mathcal{P}^+(E^n; F)$ is norm-attaining. \\
     If $E$ has the SPAP, then $(1)\iff (2) \Rightarrow (3) \iff (4).$
\end{corollary}



We showed in Section 3 that all positive bilinear operators from $\ell_4 \times \ell_4$ into $L_1([0,1])$ are norm-attaining while there are non-norm attaining bilinear operators from $\ell_4 \times \ell_4$ into $L_1([0,1])$.  Combining this example with our previous results, we obtain the corresponding weak sequential continuity dichotomy: every positive bilinear operator from $\ell_4\times \ell_4$ into $L_1([0,1])$ is weakly sequentially continuous, while some bilinear operators from $\ell_4\times \ell_4$ into $L_1([0,1])$ are not. This example, however, relies on atomic domain spaces. One might therefore wonder whether a similar phenomenon could be obtained for positive bilinear operators defined on non-atomic Banach lattices. The following example shows that, at least in its most direct form, such an extension is not possible. More precisely, positivity alone does not guarantee weak sequential continuity even for bilinear forms on non-atomic $L_p$-spaces:

\begin{example} \rm The positive bilinear form
$A:L_4[0,1]\times L_4[0,1]\longrightarrow \mathbb R$
given by
$$A(f,g)=\int_0^1 f(t)g(t)\,dt$$
is not weakly sequentially continuous. Indeed, if $(r_n)_n$ denotes the Rademacher's sequence, then $r_n \cvf 0$ in $L_4([0,1])$, but
$\displaystyle A(r_n,r_n)=\int_0^1 r_n(t)^2\,dt=1$
for every $n\in\mathbb N$. 
\end{example}

\medskip

The previous example shows that, in the non-atomic setting, positivity alone is not enough to guarantee weak sequential continuity. Thus, a natural problem is to identify structural assumptions on the domain and range Banach lattices under which positive multilinear operators still behave, with respect to weak sequential continuity, as in the atomic examples discussed above. We finish the paper with the following question.

\medskip

\noindent{\bf Question 2:}
Let $E_1,\ldots,E_n$ and $F$ be Banach lattices. Under which additional assumptions on $E_1,\ldots,E_n$ and $F$ is every positive $n$-linear operator
$A:E_1\times\cdots\times E_n\longrightarrow F$
weakly sequentially continuous?

\bigskip

\noindent{\bf Acknowledgments:} The authors would like to thank Qingying Bu for recommending reference \cite{bu2015}.

\medskip

\paragraph*{Declaration on the use of generative AI.}
The original research questions and conceptual framework of this paper were developed by the authors. OpenAI's ChatGPT was used to assist in obtaining the characterization of the sequentially positive approximation property in Proposition \ref{caracspap}, as well as to assist with language editing and improve the clarity and exposition of the manuscript. The resulting mathematical arguments and suggested revisions were carefully checked, corrected, and revised by the authors, who take full responsibility for the content of the manuscript.

\noindent L. A. Garcia\\
Instituto de Ci\^encias Exatas\\
Universidade Federal de Juiz de Fora\\
36.036-900 – Juiz de Fora – Brazil\\
e-mail: garcia.luis@ufjf.br

\medskip

\noindent J. L. P. Luiz\\
Instituto Federal do Norte de Minas Gerais\\
Campus de Ara\c cua\'i\\
39.600-00 -- Ara\c cua\'i -- Brazil\\
e-mail: lucasvt09@hotmail.com

\medskip

\noindent V. C. C. Miranda\\
Departamento de Matemática\\
Instituto de Ciências Matemáticas e de Computação \\
Universidade de São Paulo \\
13566-590 -- São Carlos - SP -- Brazil  \\
e-mail: viniciusmiranda@icmc.usp.br

\end{document}